\documentclass[12pt]{amsart}
\usepackage{pifont}
\usepackage{txfonts}
\usepackage{}
\usepackage{amsfonts}
\usepackage{graphicx}
\usepackage{epstopdf}
\usepackage{amsmath}
\usepackage{amsthm}
\usepackage{latexsym,bm}
\usepackage{amssymb}
\usepackage{esint}
\usepackage{enumitem}
\usepackage{indentfirst}
\usepackage{amscd}
\newtheorem{thm}{Theorem}[section]

\newtheorem{lem}[thm]{Lemma}
\newtheorem{prop}[thm]{Proposition}
\numberwithin{equation}{section}
\numberwithin{Remark}{section}

\begin{document}

\title{Two flow approaches to the Loewner--Nirenberg problem on manifolds}

\author{Gang Li$^\dag$}

\begin{abstract}
We introduce two flow approaches to the Loewner--Nirenberg problem on comapct Riemannian manifolds $(M^n,g)$ with boundary and establish the convergence of the corresponding Cauchy--Dirichlet problems to the solution of the Loewner--Nirenberg problem. In particular, when the initial data $u_0\in C^{4,\alpha}(M)$ is a solution or a strict subsolution to the equation $(\ref{equn_scalarcurvatureequn})$, the convergence holds for both the direct flow $(\ref{equn_scalar1})-(\ref{equn_bdv1})$ and the Yamabe flow $(\ref{equn_scalarflow2})$. Moreover, when the background metric satisfies $R_g\geq0$, the convergence holds for any positive initial data $u_0\in C^{2,\alpha}(M)$ for the direct flow; while for the case the first eigenvalue $\lambda_1<0$ for the Dirichlet problem of the conformal Laplacian $L_g$, the convergence holds for $u_0>v_0$ where $v_0$ is the largest solution to the homogeneous Dirichlet boundary value problem of $(\ref{equn_scalarcurvatureequn})$ and $v_0>0$ in $M^{\circ}$ (the interior of $M$). We also give an equivalent description between the existence of a metric of positive scalar curvature in the conformal class of $(M,g)$ and $\displaystyle\inf_{u\in C^1(M),\,u\not\equiv 0\,\text{on}\,\partial M}Q(u)>-\infty$, where $Q$ is the energy functional (see $(\ref{equn_energy123-1-1})$) of the second type Escobar-Yamabe problem.
\end{abstract}

\renewcommand{\subjclassname}{\textup{2000} Mathematics Subject Classification}


\thanks{$^\dag$ Research partially supported by the National Natural Science Foundation of China No. 11701326 and  the Young Scholars Program of Shandong University 2018WLJH85.}

\address{Gang Li, Department of Mathematics, Shandong University, Jinan, Shandong Province, China}
\email{runxing3@gmail.com}

\maketitle


\section{Introduction}

In the well-known paper \cite{LL}, Loewner and Nirenberg studied the blowing up boundary value problem
\begin{align*}
&\Delta u=\frac{1}{4}n(n-2)u^{\frac{n+2}{n-2}},\,\,\text{in}\,\,\Omega,\\
&u(x)\to \infty,\,\,\,\text{as}\,\,x\to \partial \Omega,
\end{align*}
with $\Omega$ a bounded domain of class $C^2$ in $\mathbb{R}^n$. They proved that there exists a unique positive solution $u$ to this problem, and there exists a constant $C>0$ depending on the domain $\Omega$ such that
\begin{align*}
|\text{dist}(x,\partial\Omega)^{\frac{n-2}{2}}u-1|\leq C \text{dist}(x,\partial\Omega)
\end{align*}
near the boundary, where $\text{dist}(x,\partial\Omega)$ is the distance of $x$ to $\partial\Omega$. This is equivalent to seeking the conformal metric $g=u^{\frac{4}{n-2}}\delta$ with $\delta$ the Euclidean metric on $\Omega$ to be geodesically complete and have constant scalar curvature $R_g=-n(n-1)$.

In \cite{AM} and \cite{AM1}, Aviles and McOwen generalized the Loewner--Nirenberg problem to compact Riemannian manifolds $(M,g)$ with boundary. Denote $M^{\circ}$ to be the interior of $M$. In particular, they considered the blowing-up Dirichlet boundary value problem
\begin{align}
&\label{equn_scalarcurvatureequn}\frac{4(n-1)}{n-2}\Delta u-R_gu-n(n-1)u^{\frac{n+2}{n-2}}=0,\,\,\,\text{in}\,\,M^{\circ},\\
&\label{equn_blowingupboundarydata}u(p)\to\infty,\,\,\,\text{as}\,\,p\to\partial M.
\end{align}
We call $(\ref{equn_scalarcurvatureequn})-(\ref{equn_blowingupboundarydata})$ the Loewner--Nirenberg problem on $(M,g)$. Using classical variational method they obtained a sequence of solutions to $(\ref{equn_scalarcurvatureequn})$ with enlarging Dirichlet boundary data that go to infinity, and using maximum principle and an integral type weak Harnack inequality, they obtained the existence of a solution $u$ to $(\ref{equn_scalarcurvatureequn})-(\ref{equn_blowingupboundarydata})$ and analyzed the asymptotic behavior of $u$ near the boundary. The uniqueness of the solution can be found in \cite{ACF} (see also \cite{GSW}) after a better understanding on the main term of the asymptotic behavior of $u$ near the boundary. For regularity of the Loewner--Nirenberg metric in the conformal class of a smooth compact manifold $(M^n,g)$ with boundary, an nice expansion of the solution near the boundary is given in \cite{ACF}\cite{Mazzeo}. Recently, Xumin Jiang and Qing Han developed a type of weighted Schauder estimates near the boundary, see \cite{HanJiang}(see also \cite{Kichenassamy}), which fits in this expansion well, and for expansion of the solution near the boundary for manifolds with corners on the boundary see \cite{HanShen} and for more references on this topic one is referred to \cite{HanJiangShen}.

 In this article, we derive two flow approaches to the Loewner--Nirenberg problem. Indeed, we introduce the Cauchy--Dirichlet problems to a direct scalar curvature flow (see $(\ref{equn_scalar1})-(\ref{equn_bdv1})$) and the Yamabe flow $(\ref{equn_scalarflow2})$ on a compact Riemannian manifold $(M,g)$ with boundary.

Let $(M^n,g)$ be a compact Riemannian manifold with boundary. Let $M^{\circ}$ be the interior of $M$. Let $k,m\geq 0$ be integers and $0<\alpha, \beta<1$. Define the function spaces
\begin{align*}
&C_{loc}^{k+\alpha,m+\beta}(\partial M\times[0,+\infty))=\{u\in C^{k+\alpha,m+\beta}(\partial M\times[0,T])\,\,\text{for any}\,\,T>0\},\,\,\text{and}\\
& C_{loc}^{k+\alpha,m+\beta}(M\times[0,+\infty))=\{u\in C^{k+\alpha, m+\beta}(M\times[0,T])\,\,\text{for any}\,\,T>0\}.
 \end{align*}
 Consider the Cauchy--Dirichlet problem
 \begin{align}
&\label{equn_scalar1}u_t=\frac{4(n-1)}{n-2}\Delta u-R_gu-n(n-1)u^{\frac{n+2}{n-2}},\,\,\,\text{in}\,\,M\times [0,+\infty),\\
&u(p,0)=u_0(p),\,\,p\in M,\\
&\label{equn_bdv1}u(q,t)=\phi(q,t),\,\,q\in\,\partial M,
\end{align}
where $u_0\in C^{2,\alpha}(M)$ and $\phi\in C_{loc}^{2+\alpha,1+\frac{\alpha}{2}}(\partial M\times[0,\infty))$ are positive functions satisfying the compatibility condition
\begin{align}\label{equn_compatibleequn1-2}
u_0(p)=\phi(p,0),\,\,\phi_t(p,0)=\frac{4(n-1)}{n-2}\Delta u_0(p)-R_gu_0(p)-n(n-1)u_0(p)^{\frac{n+2}{n-2}}\,\,\,\,\,\text{for}\,\,p\in\,\partial M,
\end{align}
and moreover, $\phi\to \infty$ as $t\to\infty$ and $\phi$ satisfies that
 \begin{align}\label{ineqn_bdvcond1}
\begin{split}
&\phi^{\frac{n}{2-n}}\frac{d\phi}{dt}\to0,\,\,\,\phi^{\frac{n-1}{2-n}}|\nabla_{g}\phi|\to 0,\\
&\phi^{\frac{n}{2-n}}|\nabla_{g}^2\phi|\to0,\,\,\text{uniformly on}\,\,\,\partial M, \,\,\text{as}\,\,t\to+\infty.
\end{split}
\end{align}
 which is called {\it the direct flow} in this paper. It is easy to check that if $\phi=e^t,\,t$, or $\log(t)$ for $t$ large, then $\phi$ satisfies the condition $(\ref{ineqn_bdvcond1})$.

 For the direct flow, we first derive the long-time existence of the flow for general initial data, see Lemma \ref{lem_longtimeexistence}. If the boundary data $\phi\to \infty$ and satisfies $(\ref{ineqn_bdvcond1})$ as $t\to\infty$, we obtain an asymptotic blowing up lower bound estimates near the boundary, see Lemma \ref{lem_asymlowbdnearbd1}. Together with the interior upper bound estimates (see Lemma \ref{lem_intupperboundestimates}, in comparison with the local upper bound estimates in \cite{AM1}) and the Harnack inequality, we obtain the uniform upper and lower bound of $u$ on any given compact subset of $M^{\circ}$. The convergence of the flow is the main part of the discussion. We have the following theorem for the case $R_g\geq0$.
\begin{thm}\label{thm_scalarpositiveupperb}
Let $(M, g)$  be a smooth compact manifold with boundary such that $R_g\geq0$. For any positive function $u_0\in C^{2,\alpha}(M)$, we can construct a positive function $\phi\in C_{loc}^{2,\alpha}(\partial M \times [0,+\infty))$ satisfying the compatibility condition $(\ref{equn_compatibleequn1-2})$
 on $\partial M\times\{0\}$ such that $\phi_t\geq0$ on $\partial M\times [t_1,\infty)$ for some $t_1>0$, and $\phi\to\infty$ as $t\to\infty$, and moreover $\phi$ satisfies $(\ref{ineqn_bdvcond1})$. For each such pair of functions $u_0$ and $\phi$, there exists a unique solution $u$ to the Cauchy--Dirichlet problem $(\ref{equn_scalar1})-(\ref{equn_bdv1})$, converging in $C_{loc}^2(M^{\circ})$ to a solution $u_{\infty}$ to the Loewner--Nirenberg problem $(\ref{equn_scalarcurvatureequn})-(\ref{equn_blowingupboundarydata})$ as $t\to +\infty$. Moreover, there exists a uniform constant $C>0$ independent of $t$ such that
\begin{align*}
\frac{1}{C}\big(\,x+\,(\inf_{p\in\partial M}\phi(p,t)\,)^{\frac{2}{2-n}}\,\big)^{\frac{2-n}{2}}-C\leq u\leq Cx^{\frac{2-n}{2}}
\end{align*}
near the boundary $\partial M$ for $t$ large, where $x$ is the distance function to the boundary.
 \end{thm}
Here we provide an example of the construction of $\phi$ in Theorem \ref{thm_scalarpositiveupperb}. For any given data $u_0\in C^{2+\alpha,1+\frac{\alpha}{2}}(M)$, by Theorem \ref{thm_bddataextension1} (this should be a standard result), there exists a function $\xi\in C^{2+\alpha,1+\frac{\alpha}{2}}(\partial M\times[0,\epsilon])$ for some $\epsilon>0$ such that $\xi$ satisfies the  compatibility condition $(\ref{equn_compatibleequn1-2})$ with the initial data $u_0$. Let $c=\inf_M u_0$. By continuity, there exists $\epsilon_1\in (0,\epsilon)$, such that $\xi(p,t)\geq \frac{c}{2}$ for $(p,t)\in \partial M\times[0,\epsilon_1]$. Pick up a smooth cut off function $\eta\in C^{\infty}(\mathbb{R})$, such that $\eta(t)=0$ for $0\leq t\leq \frac{\epsilon_1}{2}$, $\eta(t)=1$ for $t\geq \epsilon_1$, and $0\leq \eta(t)\leq 1$ for $\frac{\epsilon_1}{2}\leq t\leq \epsilon_1$. Define
\begin{align*}
\phi(p,t)=(1-\eta(t))\xi(p,t)+ \eta(t)f(t),
\end{align*}
for $(p,t)\in \partial M\times[0,\infty)$, where we can take the function $f(t)=c e^t$ for $t\geq0$, or $f(t)=c(t+1)$. It is easy to check that $\phi$ satisfies all the conditions in Theorem \ref{thm_scalarpositiveupperb}. For the more general construction of $\phi$, see the proof of Theorem \ref{thm_scalarpositiveupperb} in Page 13.

To achieve the conclusion of Theorem \ref{thm_scalarpositiveupperb}, when the initial data is a solution or a strict subsolution to $(\ref{equn_scalarcurvatureequn})$, we construct $\phi$ satisfying $(\ref{equn_compatibleequn1-2})$ which is increasing in a certain speed (see $(\ref{ineqn_bdvcond1})$) to infinity as $t\to\infty$, and we show that the solution is increasing and converges to the solution of the Loewner--Nirenberg problem, see Lemma \ref{lem_monotonocityscalar1} to Proposition \ref{prop_prop1}. In particular, $u(\cdot,t)$ is a sub-solution to $(\ref{equn_scalarcurvatureequn})$ for each $t\geq0$. Then for general positive initial data, we construct $\phi$ which satisfies $(\ref{equn_compatibleequn1-2})$ and increases in a certain speed (see $(\ref{ineqn_bdvcond1})$) for $t$ large, and we can first solve a Cauchy--Dirichlet problem with smaller initial data which is a subsolution to $(\ref{equn_scalarcurvatureequn})$, and hence the solution $u_1$ to this new problem is increasing and converges to the solution of the Loewner--Nirenberg problem; and then  we use $u_1$ to give the lower bound of the flow with general initial data by maximum principle, and we use Hamilton's technique in \cite{Hamilton} to derive the convergence of the flow. We provide examples of the compact manifolds with boundary with positive scalar curvature in Section \ref{section3}, and we obtain that on a smooth compact manifold $(M,g)$ with boundary, there exists a positive scalar curvature metric in the conformal class if and only if $\displaystyle\inf_{u\in C^1(M),\,u\not\equiv 0\,\text{on}\,\partial M}Q(u)>-\infty$, where
  \begin{align}\label{equn_energy123-1-1}
 Q(u)=\frac{\int_M(|\nabla u|^2+\frac{n-2}{4(n-1)}R_gu^2)dV_g+\frac{n-2}{2}\int_{\partial M}H_g u^2 dS}{\big(\int_{\partial M}|u|^{\frac{2(n-1)}{n-2}}dS\big)^{\frac{n-2}{n-1}}},
 \end{align}
 is the energy functional of the second type Escobar-Yamabe problem $(\ref{equn_EY2})$, see Theorem \ref{thm_pscalar1}.

For a general conformal class $(M,[g])$, we have the  well-known existence result for the Dirichlet boundary value problem of the Yamabe equation \begin{align}\label{equn_scalar2-1-0}
\begin{split}
&\frac{4(n-1)}{n-2}\Delta_g w-R_gw-n(n-1)w^{\frac{n+2}{n-2}}=0,\\
&w\big|_{\partial M}=\varphi_0,
\end{split}
\end{align}
 see \cite{ML} (also, \cite{Struwe}), the proof of which is by a direct variational method, in seek of a minimizer of the corresponding energy functional, and the key point is to show that the minimizer is a positive function which is similar to the argument in Appendix \ref{Appendix1}.
\begin{lem}\label{lem_Dirichletbvpscequn1} (\cite{ML})
Let $(M,g)$ be a compact Riemannian manifold of class $C^{k+2,\alpha}$ with boundary. For any positive function $\varphi_0\in C^{k+2,\alpha}(\partial M)$ with $k\geq0$, there exists a unique positive solution $w\in C^{k+2,\alpha}(M)$ to the Dirichlet boundary value problem $(\ref{equn_scalar2-1-0})$. In particular, $R_{w^{\frac{4}{n-2}}g}=-n(n-1)$.
\end{lem}
When $\varphi_0=0$, we show in Appendix \ref{Appendix1} that the Dirichlet boundary value problem $(\ref{equn_scalar2-1-0})$ has a nontrivial solution $v_0$ with $v_0>0$ in $M^{\circ}$ and $v_0=0$ on $\partial M$  when $\lambda_1(L_g)<0$, where $\lambda_1(L_g)$ is the first eigenvalue of the Dirichlet boundary value problem of the conformal Laplacian operator $L_g=-\frac{4(n-1)}{n-2}\Delta_g +R_g$. We can take $v_0$ to be the largest nontrivial solution. Then for any positive initial data $u_0>v_0$, by Appendix \ref{Appendix1}, there exists a solution $\bar{u}$ to $(\ref{equn_scalarcurvatureequn})$ with $v_0<\bar{u}<u_0$ on $M$. Then we take $\bar{g}=\bar{u}^{\frac{4}{n-2}}g$ as the background metric and run the direct flow and obtain the following convergence theorem.
\begin{thm}\label{thm_generaldirectflowdiscussion1}
Let $(M,g)$ be a compact Riemannian manifold with boundary of class $C^{4,\alpha}$.  Let $u_0\in C^{2,\alpha}(M)$ be any positive function such that $u_0>v_0$ on $M$, where $v_0$ is the largest solution to the homogeneous Dirichlet boundary value problem of the Yamabe equation
\begin{align*}
&\frac{4(n-1)}{n-2}\Delta u-R_gu-n(n-1)u^{\frac{n+2}{n-2}}=0,\,\,\,\text{in}\,\,M,\\
&u(p)=0,\,\,\,\text{for}\,\,p\in\partial M.
\end{align*}
 Then there exists a direct flow $g(t)$ starting from $g_0=u_0^{\frac{4}{n-2}}g$ and converging in $C_{loc}^2(M^{\circ})$ to $g_{\infty}=u_{LN}^{\frac{4}{n-2}}g$, where $u_{LN}$ the solution to the Loewner--Nirenberg problem $(\ref{equn_scalarcurvatureequn})-(\ref{equn_blowingupboundarydata})$.
\end{thm}
We remark that when the initial data is a subsolution to $(\ref{equn_scalarcurvatureequn})$, one can consider Theorem \ref{thm_scalarpositiveupperb} and Theorem \ref{thm_generaldirectflowdiscussion1} as the parabolic parallel to that of the elliptic approach to the Loewner-Nirenberg problem. Here our condition on the initial data for the convergence of the flow is pretty general. For any given data $u_0\in C^{2+\alpha,1+\frac{\alpha}{2}}(M)$ such that $u_0>v_0$, after we pick up the background metric $\bar{g}\in[g]$, we can choose $\phi$ in Theorem \ref{thm_generaldirectflowdiscussion1} in a similar way as the construction below Theorem \ref{thm_scalarpositiveupperb}, see proof of Theorem \ref{thm_generaldirectflowdiscussion1} in Page 17 and Lemma \ref{lem_genc}.

Notice that in Theorem \ref{thm_generaldirectflowdiscussion1}, $v_0=0$ by maximum principle when $R_h\geq0$ for some conformal metric $h\in[g]$. The first eigenvalue of the conformal Laplacian $L_g$ of the Dirichlet boundary value problem satisfies $\lambda_1(L_g)>0$ iff there exists $h\in[g]$ such that $R_h>0$, see Section \ref{section3}. To prove Theorem \ref{thm_generaldirectflowdiscussion1}, when the background metric $g$ satisfies $R_g=-n(n-1)$ and $u_0=1$, we choose increasing boundary data $\phi$ such that $\phi=1$, $\phi_t=\phi_{tt}=0$ on $\partial M\times\{0\}$, $\phi\to\infty$ as $t\to\infty$ and $\phi$ satisfies $(\ref{ineqn_bdvcond1})$, and show that the flow is monotone and converges to the Loewner--Nirenberg metric, see Lemma \ref{lem_generalconformalclassconvergence1}; while for general $g$ and $u_0>v_0$, let $\bar{u}$ be a solution to $(\ref{equn_scalarcurvatureequn})$ with $v_0<\bar{u}<u_0$ on $M$, and we take $\bar{g}=\bar{u}^{\frac{4}{n-2}}g$ as the background metric, and run the direct flow, see proof of Theorem \ref{thm_generaldirectflowdiscussion1} in Page 17. In particular, we let the initial data be $\eta=\frac{u_0}{\bar{u}}>1$ and construct a compatible boundary data $\phi$ satisfying $(\ref{ineqn_bdvcond1})$, and then use a solution to the Cauchy--Dirichlet problem $(\ref{equn_scalar1})-(\ref{equn_bdv1})$ with the background metric $\bar{g}$, initial data $1$ and increasing boundary data $\tilde{\phi}\leq \phi$ to give a lower bound, and we use Hamilton's technique in \cite{Hamilton} again for the upper bound control to conclude the convergence of the flow.

The (renormalized) Yamabe flow introduced by R. Hamilton is determined by the equation
 \begin{align*}
 \frac{\partial g}{\partial t}=-(R_g-\overline{R_g})g
 \end{align*}
 where $\overline{R_g}$ is the average of the scalar curvature $g$. It has been studied as an alternative approach to the Yamabe problem on closed manifolds, see \cite{Chow}\cite{Ye}\cite{SS}\cite{Brendle}\cite{Brendle1}, etc. It is also used in the study of general prescribed scalar curvature equation, see \cite{Struwe1}\cite{ChenXu}, etc. For Yamabe flow on manifolds with boundary on the Neumann type boundary problems posed by Escobar, see \cite{Brendle3}\cite{ChenHoSun}, etc. On complete non-compact Riemannian manifolds, there are also many works on the long time existence and convergence of the Yamabe flow, see \cite{CM}\cite{ML19}, etc. On compact manifolds with incomplete edge singularities Yamabe flow is studied in \cite{BV}, etc. see also \cite{Sch2}\cite{Shao} and the references there for Yamabe flow on incomplete manifolds. In \cite{Sch1}, an instantaneously complete Yamabe flow was derived in the conformal class of the hyperbolic space. In \cite{Sch2}, the author studied the relationship between the dimension of the singular sub-manifold and instantaneous completeness and incompleteness of the flow.

 Notice that the Yamabe flow is conformally covariant and it makes no difference which metric in the conformal class is chosen as the background metric of the flow. By Lemma \ref{lem_Dirichletbvpscequn1}, we take a conformal metric $g$ with $R_g=-n(n-1)$ in the conformal class $(M,[g])$ as the background metric. Now we introduce the Cauchy--Dirichlet problem of the Yamabe flow
\begin{align}\label{equn_scalarflow2}
\begin{split}
&(u^{\frac{n+2}{n-2}})_t=\frac{(n-1)(n+2)}{n-2}(\Delta_gu+\frac{n(n-2)}{4}(u-u^{\frac{n+2}{n-2}})),\\
&u(q,0)=u_0(q),\,\,q\in M,\\
&u(q,t)=\phi(q,t),\,\,(q,t)\in \partial M\times [0,+\infty),
\end{split}
\end{align}
where $\phi\to \infty$ as $t\to\infty$, to get a natural connection between a metric $g$ continuous up to the boundary on $M$ with the complete Loewner--Nirenberg metric in the conformal class. For the convergence, we have the following theorem.
\begin{thm}\label{thm_YamabeflowC4settingconvt1}
Let $(M^n,g)$ be a compact Riemannian manifold with boundary of class $C^{4,\alpha}$ and $R_g=-n(n-1)$. Let $u_0\in C^{4,\alpha}(M)$ with $u_0\geq1$ be a subsolution of the equation $(\ref{equn_scalarcurvatureequn})$ and satisfies
\begin{align}\label{inequn_comptC4increasing1-1}
L(\mu)\geq0
\end{align}
at the points $q\in \partial M$ such that $\mu=0$, where the function $\mu$ is defined as
\begin{align*}
\mu(p)=(n-1)u_0(p)^{-\frac{4}{n-2}}[\Delta_g u_0(p)+\frac{n(n-2)}{4}(u_0(p)-u_0(p)^{\frac{n+2}{n-2}})],
\end{align*}
and the linear operator $L$ is defined as
\begin{align*}
L(\mu)=\Delta_g\mu +\frac{n(n-2)}{4}(1-\frac{n+2}{n-2}u_0(p)^{\frac{4}{n-2}})\mu.
\end{align*}
Assume $\phi\in C_{loc}^{4+\alpha,2+\frac{\alpha}{2}}(M\times[0,\infty))$ satisfies the compatibility condition
  \begin{align*}
&u_0(p)=\phi(p,0),\,\,\phi(p,0)^{\frac{4}{n-2}}\phi_t(p,0)=(n-1)[\Delta_g u_0(p)+\frac{n(n-2)}{4}(u_0(p)-u_0(p)^{\frac{n+2}{n-2}})],\\
&\phi(p,0)^{\frac{4}{n-2}}\phi_{tt}(p,0)+\frac{4}{n-2}\phi(p,0)^{\frac{6-n}{n-2}}(\phi_t(p,0))^2=(n-1)L(\mu),
\end{align*}
 for $p\in \partial M$, and $\phi_t\geq0$ on $\partial M\times[0,\infty)$. Moreover, assume $\liminf_{t\to\infty}\inf_{\partial M}\phi=\infty$ and $\phi$ satisfies
  \begin{align}\label{inequn_bdcondition3-1}
\phi^{-1}\phi_t\leq \beta
\end{align}
for some constant $\beta>0$ on $\partial M\times[T,\infty)$ for some $T>0$, and
 \begin{align*}
\phi^{\frac{n-1}{2-n}}|\nabla_{g}\phi|\to 0,\,\,\,\,\phi^{\frac{n}{2-n}}|\nabla_{g}^2\phi|\to0,\,\,\text{uniformly on}\,\,\,\partial M, \,\,\text{as}\,\,t\to+\infty.
 \end{align*}
 Then there exists a unique positive solution $u$ to $(\ref{equn_scalarflow2})$ on $M\times[0,+\infty)$ with $u\in C_{loc}^{4+\alpha,2+\frac{\alpha}{2}}(M\times[0,\infty))$, and $u\to u_{LN}$ in $C_{loc}^4(M^{\circ})$ as $t\to\infty$, where $u_{LN}$ is the solution to the Loewner--Nirenberg problem $(\ref{equn_scalarcurvatureequn})-(\ref{equn_blowingupboundarydata})$.
\end{thm}
Notice that the condition $(\ref{inequn_comptC4increasing1-1})$ is to guarantee the $C^{4+\alpha,2+\frac{\alpha}{2}}$ regularity of the solution $u$ at $\partial M\times\{0\}$ and the condition $\phi_t\geq0$ for $t\geq0$, and it holds automatically if $u_0$ is a strict subsolution to $(\ref{equn_scalarcurvatureequn})$ in a neighborhood of $\partial M$, or $u_0$ is a solution to $(\ref{equn_scalarcurvatureequn})$ in a neighborhood of $\partial M$, for instance $u_0=1$ on $M$. When $\phi=\log(t),\,t,\,t^2,\,e^{t}$ or $te^t$ for $t$ large, $\phi$ satisfies $(\ref{inequn_bdcondition3-1})$ for some constant $\beta>0$. For given initial data $u_0$ in Theorem \ref{thm_YamabeflowC4settingconvt1}, we now give a construction of the boundary data $\phi$ satisfying all the requirement in the theorem: By by Theorem \ref{thm_bddataextension1} and the assumption $(\ref{inequn_comptC4increasing1-1})$, there exists a function $\xi\in C^{4+\alpha,2+\frac{\alpha}{2}}(\partial M\times[0,\epsilon])$ for some $\epsilon>0$ such that $\xi$ satisfies the  compatibility condition with the initial data $u_0$ (a subsolution to $(\ref{equn_scalarcurvatureequn})$) on $\partial M\times\{0\}$ and $\xi_t>0$ on $\partial M\times(0,\epsilon_0]$ for some $\epsilon_0\in(0,\epsilon)$. Let $c=\sup_{\partial M\times\{\epsilon_0\}} \xi$. Pick up a smooth nondecreasing cut off function $\eta\in C^{\infty}(\mathbb{R})$, such that $\eta(t)=0$ for $0\leq t\leq \frac{\epsilon_0}{2}$, $\eta(t)=1$ for $t\geq \epsilon_0$, and $0\leq \eta(t)\leq 1$ for $\frac{\epsilon_1}{2}\leq t\leq \epsilon_1$. Define
\begin{align*}
\phi(p,t)=(1-\eta(t))\xi(p,t)+ \eta(t)f(t),
\end{align*}
for $(p,t)\in \partial M\times[0,\infty)$, where we can take the function $f(t)=c C e^t$ or $f(t)=c C(t+1)$ for $t\geq0$, for some large constant $C>0$ to be determined. Now we take $C>0$ large so that $f(t)-\xi(p,t)>0$ for $0\leq t\leq \epsilon_0$ and hence $\phi_t\geq 0$ for $t\ge0$. It is easy to check that $\phi$ satisfies all the conditions in Theorem \ref{thm_YamabeflowC4settingconvt1}. By the construction of $u_0$ and $\phi$ and Theorem \ref{thm_YamabeflowC4settingconvt1}, there always exists a Yamabe flow converging to the Loewner-Nirenberg solution in each conformal class.

The strategy of the proof of Theorem \ref{thm_YamabeflowC4settingconvt1} is similar as that of Theorem \ref{thm_generaldirectflowdiscussion1}, but more is involved because of the nonlinear term on the left hand side of the Yamabe flow equation. Notice that we do not show the convergence of the Yamabe flow for general initial data, because when using Hamilton's technique in \cite{Hamilton} to get the upper bound of $u$, we are not able to show that $\limsup_{t\to\infty}\sup_{M^{\circ}}(u-u_{LN})\leq0$, although  $\sup_{M^{\circ}}(u(\cdot,t)-u_{LN}(\cdot))$ is decreasing in $t$, where $u_{LN}$ is the solution to the Loewner--Nirenberg problem $(\ref{equn_scalarcurvatureequn})-(\ref{equn_blowingupboundarydata})$.

The Cauchy-Dirchlet problem of the Yamabe flow is generalized to the Cauchy-Dirchlet problem of $\sigma_k$-Ricci flow in \cite{GL}, i.e., a flow approach to the generalized Loewner--Nirenberg problem of the fully nonlinear equations in \cite{Guan} and \cite{GSW}, where the fact the sub-solution property is preserving along the flow also plays an important role for the convergence of the flow.

It would be interesting to know whether the direct flow converges to the solution to Loewner--Nirenberg problem when $u_0<v_0$ somewhere in $M^{\circ}$ when $\lambda_1(L_g)<0$.

\vskip0.2cm
{\bf Acknowledgements.} The author would like to thank Wei Yuan, Xiaoyang Chen, Jian Ge and Professor Yuguang Shi for helpful discussion. The author would like to thank the referees for their comments which make the arguments in the paper clearer and more readable, especially for the generality of the initial and boundary data and an alternative proof of Theorem 3.2.

\section{A Direct Flow}
Let $(M^n,g)$ be a smooth compact Riemannian manifold with boundary $\partial M$, with $g\in C^{4,\alpha}$. We consider the direct flow, i.e. the Cauchy--Dirichlet problem $(\ref{equn_scalar1})-(\ref{equn_bdv1})$ with $\displaystyle\lim_{t\to +\infty}\phi(q,t)=+\infty$, uniformly. To guarantee the the solution $u$ is in $C^{2+\alpha,1+\frac{\alpha}{2}}(M\times[0,T_0])$ for some $T_0>0$ and $0<\alpha<1$, we need the compatibility condition
$(\ref{equn_compatibleequn1-2})$ with $u_0\in\,C^{2,\alpha}(M)$ and $\phi\in C^{2,\alpha}(\partial M\times[0,T])$ for all $T>0$. Moreover, in order that $u\in C^{4+\alpha,2+\frac{\alpha}{2}}(M\times[0,T_0])$, we need the additional condition
\begin{align}\label{equn_compatibleequn4-2-1}
\phi_{tt}(p,0)=L(v)(p)\,\,\,\,\,\text{for}\,\,p\in\,\partial M,
\end{align}
with $u_0\in\,C^{4,\alpha}(M)$ and $\phi\in C^{4,\alpha}(\partial M\times[0,T])$ for all $T>0$, where
\begin{align*}
v=\frac{4(n-1)}{n-2}\Delta u_0-R_gu_0-n(n-1)u_0^{\frac{n+2}{n-2}}
\end{align*}
on $M$ and $L$ is a linear operator such that
\begin{align*}
L(\varphi)=\frac{4(n-1)}{n-2}\Delta \varphi-R_g\varphi-\frac{n(n-1)(n+2)}{n-2}u_0^{\frac{4}{n-2}}\varphi
\end{align*}
for any $\varphi\in C^2(M)$.

 For later use, we now present a well-known weak Harnack inequality for parabolic inequalities, see in \cite{LP} or Peter Li's online lecture notes on geometric analysis for instance.

 \begin{lem}
 Assume $\bar{B}_{r}(p)$ is a closed geodesic $r$-ball in a complete Riemannian manifold $(M, g)$ with $r\leq 1$. Let $u\in C^{2,1}(M\times[0,T_1])$ be a positive function such that
\begin{align*}
u_t \leq\frac{4(n-1)}{(n-2)}\Delta u+C_0u,
\end{align*}
for some constant $C_0>0$. Assume that there exists a constant $C_S>0$, such that we have the Sobolev inequality
\begin{align*}
\frac{1}{|B_r(p)|}\int_{B_r(p)}|\nabla \phi|^2 \geq C_S r^{-2}\big(\frac{1}{|B_r(p)|}\int_{B_r(p)}\phi^{\frac{2n}{n-2}}\big)^{\frac{n-2}{n}},
\end{align*}
for each $\phi \in C_0^1(B_r(p))$. Then for $\theta\in(0,1)$ and $T\in(T_0,T_1)$, there exists a constant $C_1>0$ depending on $C_0$, such that
\begin{align*}
\displaystyle\sup_{q\in B_{\theta r}(p), t\in[T,T_1]}|u(q)|
\leq C_1C_S^{-\frac{n-2}{4}}&\left((1-\theta)^{-2}r^{-2}+(T-T_0)^{-1}\right)^{\frac{n^2-4}{4n}}r^{\frac{n-2}{2}}(T_1-T_0)^{\frac{n-2}{2n}}\times\\
&\big(\frac{1}{|T_1-T_0|\times|B_r(p)|}\int_{T_0}^{T_1}\int_{B_r(p)}|u|^{\frac{2n}{n-2}}\big)^{\frac{n-2}{2n}}.
\end{align*}
 \end{lem}

Remark. For a complete Riemannian manifold $(M^n, g)$, if $Ric_g\geq -k(n-1)g$ with some constant $k>0$, and assume that there exists a constant $C_v>0$ such that for a geodesic ball $B_r\subseteq M$, the volume ratio satisfies
\begin{align}\label{inequn_parab}
\frac{|B_r|}{|B_r^0|}\geq C_v,
\end{align}
where $|B_r^0|$ is the volume of a geodesic ball of radius $r$ in the space form of sectional curvature $-k$, then the Sobolev inequality holds on $B_r$ with $C_S$ depending on $k$ and $C_v$. For a compact manifold $(M,g)$ with boundary of class $C^{2,\alpha}$, such constants are uniform for each geodesic ball in the interior. Denote $M^{\circ}$ the interior of $M$, for each $p\in M^{\circ}$, we can also use classical parabolic theory in Euclidean domain to get the weak Harnack inequality. Indeed, let $r=\frac{1}{2}\min\{r_0,\,\text{dist}(p,\,\partial M)\}$ with $r_0$ the injectivity radius at $p$. In $B_{2r}(p)$, using the geodesic normal coordinates, we can write the equality as an parabolic inequality in Euclidean domain $B_r(0)$, and apply the weak Harnack theorem in \cite{Lieberman} to get the same control.

 We start with a parabolic analog to Aviles-McOwen's local upper bound estimates in \cite{AM1} on the solutions to $(\ref{equn_scalar1})$.
\begin{lem}\label{lem_intupperboundestimates}
 Assume $u>0$ is a solution to $(\ref{equn_scalar1})$ on $M\times [0,T_1)$ with $T_1>\epsilon_0$ for some $2>\epsilon_0>0$.  There exists a uniform constant $C_3>0$ depending on $\epsilon_0>0$ but independent of $T_1$ and $r$,  such that for each closed geodesic ball $\bar{B}_{2r}(p)\subseteq M^{\circ}$, we have
 \begin{align*}
 u(q,t)\leq C_3r^{-\frac{(n-2)}{2}},
 \end{align*}
 for $q\in \bar{B}_r(p)$ with $r\leq \min\{1, \sqrt{\epsilon_0}\}$, and $T_1> t\geq \frac{\epsilon_0}{2}$.
\end{lem}

\begin{proof}
Pick up a function $\varphi(q,t)=\xi(q)\eta(t)\in C^2( M\times [0,T_1))$ to be determined later. Let $\xi\geq 0$ be a cut off function with compact support $B_{2r}(p)$, so that $\xi(q)=1$ for $q\in B_r(p)$, $0\leq \xi(q)\leq 1$ for $q\in B_{2r}(p)$ and there exists a constant $C>0$ independent of $p\in M^{\circ}$ and $r$ such that $|\nabla \xi|\leq Cr^{-1}$ in $B_{2r}(p)$. Multiply $u\varphi^{\alpha}$ on both sides of $(\ref{equn_scalar1})$ and do integration on $M$, we have
 \begin{align*}
 \int_M u u_t\varphi^{\alpha}dV_g&=\int_M[\frac{4(n-1)}{n-2}u\Delta u \varphi^{\alpha}-R_gu^2\varphi^{\alpha}-n(n-1)u^{\frac{2n}{n-2}}\varphi^{\alpha}]dV_g
 \end{align*}
 Integration by parts, we obtain
 \begin{align*}
 &\frac{1}{2}\frac{d}{dt}\big(\int_M u^2\varphi^{\alpha}dV_g\big)-\frac{\alpha}{2}\int_M u^2\varphi^{\alpha-1}\varphi_t dV_g\\
 =&\int_M[-\frac{4(n-1)}{n-2}|\nabla u|^2 \varphi^{\alpha}-\frac{4(n-1)}{n-2}\alpha u\varphi^{\alpha-1}\,\nabla u\cdot\nabla \varphi-R_gu^2\varphi^{\alpha}-n(n-1)u^{\frac{2n}{n-2}}\varphi^{\alpha} ]dV_g\\
 \leq&\int_M[u^2\big(\frac{(n-1)}{n-2}\alpha^2\varphi^{\alpha-2}\,|\nabla \varphi|^2-R_g\varphi^{\alpha}\big)-n(n-1)u^{\frac{2n}{n-2}}\varphi^{\alpha} ]dV_g,
 \end{align*}
 where the last inequality is by Cauchy inequality. Therefore,
 \begin{align*}
\frac{1}{2}\frac{d}{dt}\big(\int_M u^2\varphi^{\alpha}dV_g\big)
 \leq\int_M[\big(\frac{\alpha}{2}\varphi^{\alpha-1}\varphi_t+\frac{(n-1)}{n-2}\alpha^2\varphi^{\alpha-2}\,|\nabla \varphi|^2-R_g\varphi^{\alpha}\big)\,u^2\,-\,n(n-1)u^{\frac{2n}{n-2}}\varphi^{\alpha} ]dV_g.
 \end{align*}
 Now for any $T$ such that $T_1>T\,\geq \frac{\epsilon_0}{2}$, we assume $1\geq \eta\geq 0$ with $\eta(t)=0$ for $t\leq T-\frac{r^2}{4}$, $\eta=1$ for $t\geq T-\frac{r^2}{8}$ and $|\eta'(t)|\leq \frac{12}{r^2}$ for $t>0$.   Integrate the above inequality on $t\in[T-\frac{r^2}{4},T]$ to have
  \begin{align*}
0\leq&\frac{1}{2}\big(\int_M u^2\varphi^{\alpha}dV_g\big)\big|_{t=T}\\
 \leq&\int_{T-\frac{r^2}{4}}^{T}\int_M[\big(\frac{\alpha}{2}\varphi^{\alpha-1}\varphi_t+\frac{(n-1)}{n-2}\alpha^2\varphi^{\alpha-2}\,|\nabla \varphi|^2-R_g\varphi^{\alpha}\big)\,u^2\,-\,n(n-1)u^{\frac{2n}{n-2}}\varphi^{\alpha} ]dV_gdt.
 \end{align*}
 Therefore, by H$\ddot{\text{o}}$lder inequality,
  \begin{align*}
 & n(n-1)\int_{T-\frac{r^2}{4}}^T \int_Mu^{\frac{2n}{n-2}}\varphi^{\alpha} dV_gdt\\
 \leq&\int_{T-\frac{r^2}{4}}^{T}\int_M[\big(\frac{\alpha}{2}\varphi^{\alpha-1}\varphi_t+\frac{(n-1)}{n-2}\alpha^2\varphi^{\alpha-2}\,|\nabla \varphi|^2-R_g\varphi^{\alpha}\big)\,u^2]dV_gdt\\
 \leq&[\int_{T-\frac{r^2}{4}}^T\int_M u^{\frac{2n}{n-2}}\varphi^{\alpha} dV_g dt]^{\frac{n-2}{n}}[\int_{T-\frac{r^2}{4}}^T\int_M[\big(\frac{\alpha}{2}\varphi\varphi_t+\frac{(n-1)}{n-2}\alpha^2\,|\nabla \varphi|^2-R_g\varphi^2\big)\varphi^{\alpha-2-\frac{n-2}{n}\alpha}]^{\frac{n}{2}}dV_gdt]^{\frac{2}{n}}\\
 =&[\int_{T-\frac{r^2}{4}}^T\int_M u^{\frac{2n}{n-2}}\varphi^{\alpha} dV_g dt]^{\frac{n-2}{n}}[\int_{T-\frac{r^2}{4}}^T\int_M[\big(\frac{\alpha}{2}\varphi\varphi_t+\frac{(n-1)}{n-2}\alpha^2\,|\nabla \varphi|^2-R_g\varphi^2\big)\varphi^{-2+\frac{2}{n}\alpha}]^{\frac{n}{2}}dV_gdt]^{\frac{2}{n}}.
 \end{align*}
 Taking $\alpha=n$, we have
 \begin{align*}
\big(\int_{T-\frac{r^2}{8}}^T \int_{B_r(p)}u^{\frac{2n}{n-2}}dV_gdt\big)^{\frac{2}{n}}&\leq (\int_{T-\frac{r^2}{4}}^T \int_Mu^{\frac{2n}{n-2}}\varphi^n dV_gdt)^{\frac{2}{n}}\\
&\leq\frac{1}{n(n-1)}[\int_{T-\frac{r^2}{4}}^T\int_M\big(\frac{n}{2}\varphi\varphi_t+\frac{(n-1)}{n-2}n^2\,|\nabla \varphi|^2-R_g\varphi^2\big)^{\frac{n}{2}}dV_gdt]^{\frac{2}{n}},
 \end{align*}
where the right hand side is bounded from above independent of $u$, $p$, $T$ and $r$. Recall that
 \begin{align*}
 u_t\leq \frac{4(n-1)}{(n-2)}\Delta_gu-R_gu
 \end{align*}
on $M\times[0,T_1)$. Hence on $B_r(p)\times [T-\frac{r^2}{4},T]$, combining with the weak Harnack type inequality on $B_{2r}(q)\times[T-\frac{r^2}{4},T]$ with $\theta=\frac{1}{2}$ and by the choice of $\varphi$, we have
\begin{align*}
\displaystyle\sup_{q\in B_r(p),T-\frac{r^2}{8}\leq t\leq T}u(q,t)\leq C_3r^{-\frac{n-2}{2}},
\end{align*}
for $T_1>T\geq \frac{\epsilon_0}{2}$, with a constant $C_3=C_3(\epsilon_0)>0$ independent of $p\in M^{\circ}$, $u$, $T_1$ and $r$. In particular, for any compact subset in $M^{\circ}$ and $t\geq\frac{\epsilon_0}{2}$, $u$ has a uniform upper bound independent of $t$.

\end{proof}

 For the Cauchy--Dirichlet problem $(\ref{equn_scalar1})-(\ref{equn_bdv1})$ with $\phi\to\infty$ uniformly on $\partial M$, we now estimate the asymptotic behavior of the solution $u$ near the boundary as $t\to \infty$.
\begin{lem}\label{lem_asymlowbdnearbd1}
Assume that the positive function $u\in C^{2,1}(M\times[0,\infty))$ is a solution to the Cauchy--Dirichlet problem $(\ref{equn_scalar1})-(\ref{equn_bdv1})$ with $\liminf_{t\to\infty}\inf_{\partial M}\phi=\infty$. Moreover, assume $\phi$ satisfies $(\ref{ineqn_bdvcond1})$. Then there exist constants $C>0$, $t_1>0$ large and $x_1>0$ small such that
\begin{align*}
u\geq \frac{1}{C}\big(\,x+\,(\inf_{p\in\partial M}\phi(p,t)\,)^{\frac{2}{2-n}}\,\big)^{\frac{2-n}{2}}-C,
\end{align*}
for $x\leq x_1$ and $t\geq t_1$, where $x=x(p)$ is the distance function to the boundary.
\end{lem}
\begin{proof}
 Recall that there exists a conformal metric $h$ with $g=f^{\frac{4}{n-2}}h$ so that $R_h>0$ in a neighborhood $V$ of the boundary $\partial M$, where $f\in C^{2,\alpha}(M)$ is a positive function, see in Lemma 3.1 in \cite{ACF}. Indeed, for that one can use the formula of the scalar curvature under conformal change and do Taylor expansion of $f$ in the direction of $x$ in a small neighborhood of the boundary and just take $f$ to be a quadratic polynomial of $x$ with coefficients functions on $\partial M$. Let $x$ be the distance function to the boundary $\partial M$ under $h$. Take $x_1>0$ small so that the boundary neighborhood $U=\{0\leq x\leq x_1\}$ lies in $V$ and hence $R_h>0$, and $U$ is diffeomorphic to $\partial M\times [0,x_1]$ under the exponential map $F:\,\partial M\times [0,x_1]\to U$, where $F(q, x)=\text{Exp}_q^h(x)\in U$ is on the geodesic starting from $q\in \partial M$ in $(V,h)$ in the inner normal direction of $\partial M$ of distance $x$ to $q$. We define a function $\varphi\in C_{loc}^{2,\alpha}(U\times [0,+\infty))$ such that
 \begin{align}\label{equn_barrierfunctionlowerbound1}
 \varphi(\text{Exp}^h_q(x),t)=c[(x+(f(q)\phi(q,t))^{\frac{2}{2-n}})^{\frac{2-n}{2}}-(x_1+(f(q)\phi(q,t))^{\frac{2}{2-n}})^{\frac{2-n}{2}}]
  \end{align}
  for $(q,x,t)\in \partial M\times[0,x_1]\times[0,+\infty)$, where $c>0$ is a small constant to be determined. Let $\tilde{u}=fu$. Then we have
 \begin{align*}
 &\tilde{u}_t=f^{-\frac{4}{n-2}}[\frac{4(n-1)}{n-2}\Delta_h \tilde{u}-R_h\tilde{u}-n(n-1)\tilde{u}^{\frac{n+2}{n-2}}],\,\,\,\text{in}\,\,M\times [0,+\infty),\\
&\tilde{u}(p,0)=f(p)u_0(p),\,\,p\in M,\\
&\tilde{u}(q,t)=f(q)\phi(q,t),\,\,q\in\,\partial M.
 \end{align*}

 Note that $h$ has the orthogonal decomposition $h=dx^2+h_x$ on $U$ with $h_0=h\big|_{\partial M}$. Then by direct calculation, one has
\begin{align}\label{equn_lowerboundnearboundary}
\begin{split}
\Delta_h\varphi&=\partial _x^2\varphi+\frac{1}{2}H_{x}\partial_x\varphi+\Delta_{h_x}\varphi,\\
\varphi_t&=cf^{\frac{2}{n-2}}[(x+(f(q)\phi(q,t))^{\frac{2}{2-n}})^{\frac{-n}{2}}-(x_1+(f(q)\phi(q,t))^{\frac{2}{2-n}})^{\frac{-n}{2}}]\phi^{\frac{n}{2-n}}\frac{d\phi}{dt}\\
\partial_x\varphi&=\frac{2-n}{2}c(x+(f(q)\phi(q,t))^{\frac{2}{2-n}})^{\frac{-n}{2}},\,\,\,\,\partial_x^2\varphi=\frac{n(n-2)}{4}c(x+(f(q)\phi(q,t))^{\frac{2}{2-n}})^{\frac{-n-2}{2}},\\
|\nabla_{h_x}\varphi|&\leq c(x+(f(q)\phi(q,t))^{\frac{2}{2-n}})^{\frac{-n}{2}}(f\phi)^{\frac{n}{2-n}}|\nabla_{h_x}(f\phi)|,\\
|\nabla_{h_x}^2\varphi|&\leq c(x+(f(q)\phi(q,t))^{\frac{2}{2-n}})^{\frac{-n}{2}}\,\,\big[(f\phi)^{\frac{n}{2-n}}|\nabla_{h_x}^2(f\phi)|+\frac{n}{n-2}(f\phi)^{\frac{2n-2}{2-n}}|\nabla_{h_x}(f\phi)|^2\\
&\,\,\,\,+\frac{n}{n-2}(x+(f(q)\phi(q,t))^{\frac{2}{2-n}})^{-1}(f\phi)^{\frac{2n}{2-n}}|\nabla_{h_x}(f\phi)|^2\big],
\end{split}
\end{align}
where $H_{s}$ is the mean curvature of the hypersurface $\Sigma_s=\{x=s\}$. Recall that $\phi\to\infty$ uniformly on $\partial M$ as $t\to\infty$. By the assumption $(\ref{ineqn_bdvcond1})$, we have that there exists $t_1\geq 0$ such that for $t\geq t_1$,
\begin{align*}
\varphi_t\leq f^{-\frac{4}{n-2}}[\frac{4(n-1)}{(n-2)}\Delta_h\varphi-R_h\varphi-n(n-1)\varphi^{\frac{n+2}{n-2}}]
\end{align*}
on $U$, since all other terms are lower order terms as $t\to\infty$ in comparison with the terms
\begin{align}\label{equn_maintermtestbdr}
\begin{split}
&\frac{4(n-1)}{(n-2)}\Delta_h\varphi-n(n-1)\varphi^{\frac{n+2}{n-2}}=n(n-1)c(x+(f(q)\phi(q,t))^{\frac{2}{2-n}})^{\frac{-n-2}{2}}(1+o(1))\\
&-n(n-1)c^{\frac{n+2}{n-2}}[(x+(f(q)\phi(q,t))^{\frac{2}{2-n}})^{\frac{2-n}{2}}-(x_1+(f(q)\phi(q,t))^{\frac{2}{2-n}})^{\frac{2-n}{2}}]^{\frac{n+2}{n-2}},
 \end{split}
 \end{align}
 with $c>0$ small and the term $o(1)\to 0$ uniformly on $U$ as $t\to +\infty$. Let $v=\tilde{u}-\varphi$. Therefore,
\begin{align*}
v_t\geq f^{\frac{-4}{n-2}}[\frac{4(n-1)}{(n-2)}\Delta_h v-\,\big(R_h+n(n-1)\xi(p,t)\,\big)\,v\,]
\end{align*}
on $U\times[t_1,+\infty)$, where $\xi(p,t)=\frac{u^{\frac{n+2}{n-2}}-\varphi^{\frac{n+2}{n-2}}}{u-\varphi}$ when $u\neq \varphi$, and $\xi=\frac{n+2}{n-2}u^{\frac{4}{n-2}}$ otherwise, and hence, $\xi(p,t)>0$ on $U\times [t_1,+\infty)$.  Since $\tilde{u}$ is continuous on $U\times \{t=t_1\}$, we choose $c>0$ small so that $\tilde{u}>\varphi$ on $U\times\{t=t_1\}$. On the other hand, $\varphi=0<\tilde{u}$ on $\{x=x_1\}\times[t_1,\infty)$ and by definition $\varphi\leq \tilde{u}$ on $\partial M$ and hence, $v\geq0$ on $\partial U\times [t_1,+\infty)\bigcup U\times\{t_1\}$. Recall that $R_h>0$ on $U$, and hence by maximum principle,
\begin{align*}
v\geq0
\end{align*}
on $U\times [t_1,+\infty)$. Therefore,
\begin{align*}
u\geq f^{-1}\varphi
\end{align*}
on $U\times [t_1,+\infty)$. Combining this lower bound estimate of $u$ near the boundary with the Harnack inequality $(\ref{inequn_Harnack1})$ based on the uniform interior upper bound estimates in Lemma \ref{lem_intupperboundestimates}, and also using finite cover of geodesic balls, on any compact sub-domain of $M^{\circ}$ we obtain uniform positive lower bound of $u$ for $t\in[t_1,+\infty)$.

\end{proof}

Now we give the long time existence of the flow.
\begin{lem}\label{lem_longtimeexistence}
There exists a unique solution $u \in C_{loc}^{2+\alpha,1+\frac{\alpha}{2}}(M\times[0,\infty))$ to the Cauchy--Dirichlet boundary $(\ref{equn_scalar1})-(\ref{equn_bdv1})$ with positive functions $u_0\in C^{2,\alpha}(M)$ and $\phi\in C_{loc}^{2+\alpha,1+\frac{\alpha}{2}}(\partial M\times[0,\infty))$ satisfying $(\ref{equn_compatibleequn1-2})$. Moreover, if in addition, $u_0\in C^{4,\alpha}(M)$ and $\phi\in C_{loc}^{4+\alpha,2+\frac{\alpha}{2}}(\partial M\times[0,\infty))$ satisfy $(\ref{equn_compatibleequn4-2-1})$, then $u\in C_{loc}^{4+\alpha,2+\frac{\alpha}{2}}(M\times[0,\infty))$. Moreover, if $\liminf_{t\to\infty}\inf_{\partial M}\phi=\infty$ and $\phi$ satisfies $(\ref{ineqn_bdvcond1})$ as $t\to\infty$, then there exists constants $C>0$, $t_2>0$ large and $x_1>0$ small such that
\begin{align}\label{inequn_bdlbest1-1}
u\geq \frac{1}{C}\big(\,x+\,(\inf_{p\in\partial M}\phi(p,t)\,)^{\frac{2}{2-n}}\,\big)^{\frac{2-n}{2}}-C,
\end{align}
for $x\leq x_1$ and $t\geq t_2$, where $x=x(p)$ is the distance function to the boundary, and for each compact subset $F\subseteq M^{\circ}$, there exists $C=C(F)>0$ such that
\begin{align*}
u(q,t)\geq C
\end{align*}
for each $(q,t)\in F\times[0,\infty)$. Moreover, for each compact subset $F\subseteq M^{\circ}$, there exists $m=m(F)>0$ independent of $u_0$ and $\phi$ such that
\begin{align*}
\liminf_{t\to\infty}\inf_{q\in F}u(q,t)\geq m,
\end{align*}
provided that $\liminf_{t\to\infty}\inf_{\partial M}\phi=\infty$.
\end{lem}
\begin{proof}
Since $(\ref{equn_scalar1})$ is parabolic, by the compatibility condition and regularity of the positive functions $u_0$ and $\phi$, classical theory for semilinear parabolic equations (see Theorem 8.2 in \cite{Lieberman}, the proof of which is based on the Schauder fixed point theorem) yields the existence of a unique positive solution $u$ to $(\ref{equn_scalar1})-(\ref{equn_bdv1})$ on $M\times[0,T_0)$ for some $T_0>0$ with $u$ satisfying $u\in C^{2+\alpha,1+\frac{\alpha}{2}}(M\times[0,T])$ for each $0<T<T_0$ when $u_0\in C^{2,\alpha}(M)$ and $\phi\in C_{loc}^{2+\alpha,1+\frac{\alpha}{2}}(M\times[0,\infty))$ satisfying $(\ref{equn_compatibleequn1-2})$; moreover, $u \in C^{4+\alpha,2+\frac{\alpha}{2}}(M\times[0,T])$ for each $0<T<T_0$ when $u_0\in C^{4,\alpha}(M)$ and $\phi\in C^{4+\alpha,2+\frac{\alpha}{2}}$ satisfy $(\ref{equn_compatibleequn4-2-1})$ in addition. We assume $T_0$ is the maximum time for the existence of the positive solution $u$ on $[0,T_0)$. By maximum principle, we have the upper bound
\begin{align*}
\sup_{q\in M,\,t\in[0,T)}u \leq \max\{\sup_{q\in\partial M,\,t\in[0,T)}\phi,\,\sup_{q\in M}u_0(q),\,\big(\frac{1}{n(n-1)}\max\{-\inf_{q\in M}R_g(q),0\}\big)^{\frac{n-2}{4}}\},
\end{align*}
for any $0<T<T_0$, which is uniformly bounded on $M\times[0,T_0)$. Then standard Schauder theory derives that there exists $C=C(T_0)>0$ such that
\begin{align*}
\|u\|_{C^{2+\alpha,1+\frac{\alpha}{2}}(M\times[0,T])}\leq C
\end{align*}
in the $C^{2+\alpha,1+\frac{\alpha}{2}}$ setting Cauchy--Dirichlet problem, for each $0<T<T_0$; while
\begin{align*}
\|u\|_{C^{4+\alpha,2+\frac{\alpha}{2}}(M\times[0,T])}\leq C
\end{align*}
in the $C^{4+\alpha,2+\frac{\alpha}{2}}$ setting Cauchy--Dirichlet problem, for each $0<T<T_0$. To establish the lower bound estimates, for any $q\in M^{\circ}$, assume $0<r<\frac{\text{dist}(q,\partial M)}{2}$ and $r<1$. Then by the Harnack inequality, there exists a constant $C>0$ depending on $B_{2r}(q)$ but independent of $T$ such that
\begin{align}\label{inequn_Harnack1}
\sup_{ B_r(q)\times[T-\frac{3r^2}{4}, T-\frac{r^2}{2}]}u\leq C \inf_{B_r(q)\times[T-\frac{r^2}{4}, T]}u
\end{align}
for each $r^2<T<T_0$.   By Lemma \ref{lem_intupperboundestimates}, $|R_g+n(n-1)u^{\frac{4}{n-2}}|$ is uniformly bounded in $B_{2r}(q)$ (independent of $T_0$) and hence, the constant $C$ in $(\ref{inequn_Harnack1})$ is independent of $T_0$. Recall that $\phi>0$ on $\partial M \times[0,\infty)$. Therefore, $u$ can be extended to a positive solution of class $C_{loc}^{2+\alpha,1+\frac{\alpha}{2}}(M\times[0,\infty))$ (resp. of class $C_{loc}^{4+\alpha,2+\frac{\alpha}{2}}(M\times[0,\infty))$).

Assume that $\liminf_{t\to\infty}\inf_{\partial M}\phi=\infty$ and $\phi$ satisfies $(\ref{ineqn_bdvcond1})$ as $t\to\infty$. Then by Lemma \ref{lem_asymlowbdnearbd1} we have that there exist $x_1>0$, $C>0$ and $t_2>0$ such that $(\ref{inequn_bdlbest1-1})$ holds in $\{0\leq x\leq x_1\}$ for $t\geq t_2$. Using this uniform lower bound estimate near the boundary, combining with the above Harnack inequality and a finite geodesic-ball cover of the path connecting to $\{0\leq x\leq x_1\}$, we have that for each compact subset $F\subseteq M^{\circ}$, there exists $C=C(F)>0$ such that
\begin{align*}
u(q,t)\geq C
\end{align*}
for $(q,t)\in F\times[0,\infty)$, and hence by Lemma \ref{lem_intupperboundestimates}, $u$ is uniformly bounded from above and below by positive constants in $F\times[0,\infty)$.

On any compact subset $F\subseteq M^{\circ}$, since the upper bound of $u$ obtained in Lemma \ref{lem_intupperboundestimates} is independent of $u_0$ and $\phi$, and so is $|R_g+n(n-1)u^{\frac{4}{n-2}}|$ and hence the constant $C$ in the Harnack inequality is independent of $u_0$ and $\phi$. Therefore, by Lemma \ref{lem_asymlowbdnearbd1} there exists $m>0$ depending on $F$ but independent of $u_0$ and $\phi$ such that
\begin{align*}
\liminf_{t\to\infty}\inf_{q\in F}u(q,t)\geq m.
\end{align*}

\end{proof}

Now we consider the convergence of the flow. To warm up, we first consider the case $R_{g}\geq0$.
\begin{prop}\label{prop_prop1}
Let $(M^n,g)$ be a compact manifold with $g\in C^{4,\alpha}$ up to the boundary such that $R_{g}\geq 0$. Assume $u_0\in C^{4,\alpha}(M)$ and $\phi\in C_{loc}^{4+\alpha,2+\frac{\alpha}{2}}(M\times[0,\infty))$ satisfy $(\ref{equn_compatibleequn1-2})$ and $(\ref{equn_compatibleequn4-2-1})$, and also $\phi_t\geq0$ on $M\times[0,\infty)$, $\liminf_{t\to\infty}\inf_{\partial M}\phi=\infty$ and $\phi$ satisfies $(\ref{ineqn_bdvcond1})$ as $t\to\infty$. Moreover, assume $u_0$ is a positive subsolution to $(\ref{equn_scalarcurvatureequn})$ and
 \begin{align}\label{inequn_compatibleincreasingcondn}
 L(v)(p)\geq0
 \end{align}
 for any $p\in \partial M$ such that $v(p)=0$, where $L$ and $v$ are defined in $(\ref{equn_compatibleequn4-2-1})$. Then we have that the Cauchy--Dirichlet problem $(\ref{equn_scalar1})-(\ref{equn_bdv1})$ has a unique global solution $u$, which converges in $C_{loc}^4(M^{\circ})$ to the unique solution $u_{\infty}$ to the Loewner--Nirenberg problem.
\end{prop}
Remark that for a sub-solution $u_0$ of $(\ref{equn_scalarcurvatureequn})$, which is a strict sub-solution in a neighborhood of $\partial M$, the condition $(\ref{inequn_compatibleincreasingcondn})$ disappears automatically. For instance, if $\varphi>0$ is a subsolution of $(\ref{equn_scalarcurvatureequn})$, then $\epsilon \varphi$ is a strict sub-solution for any constant $0<\epsilon<1$. A solution $\varphi>0$ to $(\ref{equn_scalarcurvatureequn})$ automatically satisfies $(\ref{inequn_compatibleincreasingcondn})$. Notice that the functions $\log(t),\,\,t$  and $e^t$ all satisfy $(\ref{ineqn_bdvcond1})$. Also, if $u_0$ is a solution to $(\ref{equn_scalarcurvatureequn})$ in a neighborhood of $\partial M$, $(\ref{inequn_compatibleincreasingcondn})$ holds automatically. To prove this proposition, we need two lemmas.

 \begin{lem}\label{lem_monotonocityscalar1}
Let $(M, g)$  be a smooth compact manifold with boundary such that $R_g\geq0$. Assume $\phi\in C^{4}(\partial M \times [0,+\infty))$ is a positive function with $\phi_t\geq0$, and $u_0\in C^4(M)$ is a positive function which is a subsolution to $(\ref{equn_scalarcurvatureequn})$ on $M$. Assume that a positive function $u\in C^{4,2}(M\times[0,\infty))$ is a solution to the Cauchy--Dirichlet problem $(\ref{equn_scalar1})-(\ref{equn_bdv1})$. Then $u$ satisfies that $u_t\geq0$ in $M^{\circ}\times (0,+\infty)$. That's to say, for any $t>0$, $u(\cdot,t)$ is a subsolution to the Yamabe equation $(\ref{equn_scalarcurvatureequn})$.
 \end{lem}
  \begin{proof}
  The proof is an application of the maximum principle to the equation satisfied by $u_t$. Indeed, denote $v=u_t$. It is clear that $v\in C^{2,1}(M\times [0,+\infty))$. We take derivative of $t$ on both sides of $(\ref{equn_scalar1})$ and $(\ref{equn_bdv1})$ to obtain the Cauchy--Dirichlet problem
  \begin{align}
  &\label{equn_linearequation1-1}v_t=\frac{4(n-1)}{n-2}\Delta v-R_gv-\frac{n(n-1)(n+2)}{n-2}u^{\frac{4}{n-2}}v,\,\,\,\text{in}\,\,M\times [0,+\infty),\\
&v(p,0)=u_t(p,0)=\frac{4(n-1)}{n-2}\Delta u_0(p)-R_gu_0(p)-n(n-1)u_0(p)^{\frac{n+2}{n-2}}\geq0,\,\,p\in M,\nonumber\\
&v(q,t)=\phi_t(q,t)\geq0,\,\,q\in\,\partial M.\nonumber
  \end{align}
  Therefore, by maximum principle, $v=u_t\geq0$ on $M\times[0,+\infty)$. Moreover, if there exists a constant $t_1>0$ such that $\phi_t>0$ for $t>t_1$, then by strong maximum principle, we have $v=u_t>0$ in $M^{\circ}\times(t_1,+\infty)$.
  \end{proof}

Proposition \ref{prop_prop1} is a direct consequence of the following lemma.
 \begin{lem}
Let $(M, g)$, $u_0$ and $\phi$ be as in Proposition \ref{prop_prop1}. Then the unique solution $u$ to the Cauchy--Dirichlet problem $(\ref{equn_scalar1})-(\ref{equn_bdv1})$ converges in $C_{loc}^{4}(M^{\circ})$ to the solution $u_{\infty}$ of the Loewner--Nirenberg problem $(\ref{equn_scalarcurvatureequn})-(\ref{equn_blowingupboundarydata})$ as $t\to +\infty$. Moreover, there exists a constant $C>0$ such that $\frac{1}{C}x^{\frac{2-n}{2}}\leq u_{\infty}\leq Cx^{\frac{2-n}{2}}$ near the boundary $\partial M$, where $x$ is the distance function to the boundary.
 \end{lem}
 \begin{proof}
 By Lemma \ref{lem_longtimeexistence}, there exists a unique positive solution $u\in C_{loc}^{4+\alpha,2+\frac{\alpha}{2}}(M\times[0,\infty))$ to $(\ref{equn_scalar1})-(\ref{equn_bdv1})$. By Lemma \ref{lem_monotonocityscalar1}, $u_t\geq0$ on $M\times[0,\infty)$. Since we have the upper and lower bound estimates on $u$ on any compact sub-domain of $M^{\circ}$ and $u$ is increasing pointwisely on $M^{\circ}$ as $t$ increases, we have that $u$ converges pointwisely in $M^{\circ}$ as $t\to+\infty$. By the Harnack inequality for $v=u_t$ in the parabolic equation $(\ref{equn_linearequation1-1})$, we have that $u_t\to0$ locally uniformly and hence $u$ converges locally uniformly in $M^{\circ}$ to a positive function $u_{\infty}$. Using classical parabolic estimates, we have that $u(\cdot, t)\to u_{\infty}$ in $C_{loc}^4(M)$ as $t\to +\infty$. Therefore, $u_{\infty}$ is a solution to $(\ref{equn_scalarcurvatureequn})$ in $M^{\circ}$. By the upper bound and lower bound estimates of $u$ near the boundary $\partial M$ in Lemma \ref{lem_intupperboundestimates} and Lemma \ref{lem_asymlowbdnearbd1}, we obtain the estimate of $u_{\infty}$ near the boundary as stated in the lemma, and hence $u_{\infty}$ is the unique solution to the Loewner--Nirenberg problem. This completes the proof of the lemma and Proposition \ref{prop_prop1}.
 \end{proof}

  \begin{proof}[Proof of Theorem \ref{thm_scalarpositiveupperb}]
  We first construct the boundary data $\phi\in C_{loc}^{2+\alpha,1+\frac{\alpha}{2}}(\partial M\times [0,\infty))$. Pick up a positive number $t_1>0$. By Theorem \ref{thm_bddataextension1}, we let $\phi$ be a positive function and satisfy the compatibility condition $(\ref{equn_compatibleequn1-2})$ at $t=0$, $\phi_t\geq 0$ for $t\geq t_1$, $\phi\to \infty$ as $t\to \infty$, and let $\phi$ satisfy $(\ref{ineqn_bdvcond1})$ as $t\to\infty$ (see the remark after the statement of Proposition \ref{prop_prop1}). Since $u_0$ is not necessarily a subsolution of $(\ref{equn_scalarcurvatureequn})$, $\phi_t$ is not necessarily positive at $t=0$. We let $\phi>0$ for $t\in[0,t_1]$. By Lemma \ref{lem_longtimeexistence}, there exists a unique positive solution $u\in C_{loc}^{2+\alpha,1+\frac{\alpha}{2}}(M\times[0,\infty))$ to $(\ref{equn_scalar1})-(\ref{equn_bdv1})$.

  By Lemma \ref{lem_Dirichletbvpscequn1}, for any $\epsilon>0$, there exists a unique positive solution $w\in C^{4,\alpha}(M)$ to the Dirichlet boundary value problem
  \begin{align*}
  &\frac{4(n-1)}{n-2}\Delta w-R_gw-n(n-1)w^{\frac{n+2}{n-2}}=0,\\
  &w\big|_{\partial M}=\epsilon.
  \end{align*}
  By maximum principle, $w\leq \epsilon$.

 Let $\epsilon>0$ be a small constant so that $\epsilon< \frac{1}{10n} \inf_{\partial M\times[0,t_1]}\phi$. It is easy to choose a boundary data $\tilde{\phi}(q,t)\in C_{loc}^{4+\alpha,2+\frac{\alpha}{2}}(\partial M \times [0,+\infty))$ such that $\tilde{\phi}=\epsilon$ and $\tilde{\phi}_t=\tilde{\phi}_{tt}=0$ on $\partial M\times\{0\}$, moreover, $\tilde{\phi}_t\geq0$ and $\tilde{\phi}\leq \phi$ on $\partial M\times [0,+\infty)$. Also, we require that $\tilde{\phi}\to\infty$ and $\tilde{\phi}$ satisfies $(\ref{ineqn_bdvcond1})$ as $t\to\infty$ (for this we can just let $\tilde{\phi}=\phi$ for $t$ large). Therefore, by Proposition \ref{prop_prop1} there exists a unique solution $\tilde{u}$ to the problem
  \begin{align*}
&\tilde{u}_t=\frac{4(n-1)}{n-2}\Delta \tilde{u}-R_g\tilde{u}-n(n-1)\tilde{u}^{\frac{n+2}{n-2}},\,\,\,\text{in}\,\,M\times [0,+\infty),\\
&\tilde{u}(p,0)=w,\,\,p\in M,\\
&\tilde{u}(q,t)=\tilde{\phi}(q,t),\,\,q\in\,\partial M,
\end{align*}
 such that $\tilde{u}$ is increasing in $t$, and converges to the Loewner--Nirenberg solution $u_{\infty}$. Let $v=u-\tilde{u}$. Then $v$ satisfies
 \begin{align*}
 &v_t=\frac{4(n-1)}{n-2}\Delta v-(R_g+n(n-1)\zeta)v,\,\,\,\text{in}\,\,M\times [0,+\infty),\\
&v(p,0)=u_0(p)-w(p)>0,\,\,p\in M,\\
&v(q,t)=\phi(q,t)-\tilde{\phi}(q,t)\geq0,\,\,q\in\,\partial M,
 \end{align*}
where $\zeta=\frac{u^{\frac{n+2}{n-2}}-\tilde{u}^{\frac{n+2}{n-2}}}{u-\tilde{u}}$ for $u\neq \tilde{u}$, and $\zeta=\frac{n+2}{n-2}u^{\frac{4}{n-2}}$ otherwise. In particular, $\zeta>0$. By maximum principle, $v\geq 0$ on $M\times[0,+\infty)$.

 Now we establish the upper bound of $u$. Let $\xi(p,t)=u(p,t)-u_{\infty}(p)$, with $u_{\infty}$ the Loewner--Nirenberg solution. Therefore, $\xi(p,t)\to-\infty$ as $p\to\partial M$ for each $t>0$. Hence, for any $t>0$, there exists $p_t\in M^{\circ}$ such that $\xi(p_t,t)=\sup_{p\in M^{\circ}}\xi(p,t)$, and hence $\Delta_g\xi(p_t,t)\leq 0$. Therefore, we have
 \begin{align}\label{inequn_upperb1}
 \xi_t&=\frac{4(n-1)}{n-2}\Delta \xi-(R_g+n(n-1)\mu(x,t))\xi\\
 &\leq -(R_g+n(n-1)\mu)\xi,\nonumber
 \end{align}
at the point $(p_t,t)$ for $t>0$, where $\mu(x,t)=\frac{u^{\frac{n+2}{n-2}}-u_{\infty}^{\frac{n+2}{n-2}}}{u-u_{\infty}}$ for $u\neq u_{\infty}$, and $\mu=\frac{n+2}{n-2}u_{\infty}^{\frac{4}{n-2}}$ otherwise. In particular, $\mu>0$. By maximum principle for the equation $(\ref{inequn_upperb1})$ satisfied by $\xi$, if $\xi(p_{t_1},t_1)\leq0$ for some $t_1>0$, then $\xi(p,t)\leq 0$ for all $t>t_1$. Therefore, there are two possibilities: one is that there exists $t_2>0$ so that for $t>t_2$ we have $\xi\leq 0$ on $M^{\circ}$; the second case is that $\xi(p_t,t)>0$ for $t\in[0,+\infty)$. For the second case, let $\eta(t)=\sup_{p\in M^{\circ}}\xi(p,t)$. By the inequality in $(\ref{inequn_upperb1})$, $\eta(t)$ is decreasing since $\eta(t)>0$. Now we use the discussion as in \cite{Hamilton} to show that $\displaystyle\limsup_{t\to+\infty}\eta(t)\leq 0$. Denote $(\frac{d\eta}{dt})_+(t)=\displaystyle\limsup_{\tau\searrow 0}\frac{\eta(t+\tau)-\eta(t)}{\tau}$ and define the set
 \begin{align*}
 S(t)=\{p\in M^{\circ}\big|\, p\,\,\text{is a maximum point of }\,\xi\,\,\text{on}\,\,M^{\circ}\times\{t\}\}.
  \end{align*}
 Then $S(t)$ is compact for any $t\geq0$. By Lemma 3.5 in \cite{Hamilton},
 \begin{align*}
 (\frac{d\eta}{dt})_+(t)\leq \sup\{\xi_t(p_t,t)\big|\,p_t\in\, S(t)\},
  \end{align*}
  since $\eta(t)>0$. By the mean value theorem we have that $\frac{u^{\frac{n+2}{n-2}}-u_{\infty}^{\frac{n+2}{n-2}}}{u-u_{\infty}}\geq \frac{n+2}{n-2}(\displaystyle\inf_{q\in M}u_{\infty})^{\frac{4}{n-2}}$. Therefore, for the second case,
 \begin{align*}
 (\frac{d\eta}{dt})_+(t)\leq -(R_g+\frac{n(n-1)(n+2)}{n-2}(\displaystyle\inf_{q\in M}u_{\infty})^{\frac{4}{n-2}})\eta(t),
 \end{align*}
 for $t\geq0$. By integration or a contradiction argument we can easily obtain that $\displaystyle\limsup_{t\to+\infty}\eta(t)\leq 0$. In summary, for both cases we obtain the upper bound control that
 \begin{align*}
 \displaystyle\limsup_{t\to+\infty}\sup_{q\in M^{\circ}}(u(q,t)-u_{\infty}(q))\leq0.
 \end{align*}
 Based on the upper bound and lower bound control on $u$, we have that $\displaystyle\lim_{t\to+\infty}u(q,t)\to u_{\infty}(q)$ locally uniformly on $M^{\circ}$ as $t\to+\infty$. By the standard interior parabolic estimates, we obtain that the convergence is in $C_{loc}^2(M^{\circ})$ sense.
  \end{proof}

Let $(M,g)$ be a general compact Riemannian manifold of class $C^{4,\alpha}$ with boundary. By Lemma \ref{lem_Dirichletbvpscequn1}, there exists a conformal metric $h\in[g]$ of class $C^{4,\alpha}$ such that $R_h=-n(n-1)$. We still denote it as $g$ and choose it as the background metric of our flow,
and hence the Cauchy--Dirichlet problem $(\ref{equn_scalar1})-(\ref{equn_bdv1})$ becomes
\begin{align}\label{equn_scalar2-1}
\begin{split}
&u_t=\frac{4(n-1)}{n-2}\Delta_{g} u+n(n-1)\big(u-u^{\frac{n+2}{n-2}}\big),\,\,\,\text{in}\,\,M\times [0,+\infty),\\
&u(p,0)=u_0(p),\,\,p\in M,\\
&u(q,t)=\phi(q,t),\,\,q\in\,\partial M,
\end{split}
\end{align}
with $\phi> 0$ and $u_0>0$ compatible, and $\phi\to +\infty$ uniformly as $t\to+\infty$. Recall that we have a unique positive solution $u\in C_{loc}^{4+\alpha,2+\frac{\alpha}{2}}(M\times[0,+\infty))$ to the Cauchy--Dirichlet boundary value problem with uniform upper and lower bound estimates on any compact sub-domain in $M^{\circ}$ for $t\in[0,+\infty)$, and the upper bound and lower bound asymptotic behavior estimates near the boundary as $t\to+\infty$. We now consider the convergence. To insure the asymptotic behavior near the boundary as $t\to+\infty$, we always assume the boundary data $\phi$ satisfies $(\ref{ineqn_bdvcond1})$ for $t$ large.

We start with the special initial data $u_0=1$.
\begin{lem}\label{lem_generalconformalclassconvergence1}
Assume $(M,g)$ is a compact Riemannian manifold with boundary of class $C^{4,\alpha}$ and $R_g=-n(n-1)$. Let $u>0$ be the solution to the Cauchy--Dirichlet problem
\begin{align}\label{equn_scalar20}
\begin{split}
&u_t=\frac{4(n-1)}{n-2}\Delta_{g} u+n(n-1)\big(u-u^{\frac{n+2}{n-2}}\big),\,\,\,\text{in}\,\,M\times [0,+\infty),\\
&u(p,0)=1,\,\,p\in M,\\
&u(q,t)=\phi(q,t),\,\,q\in\,\partial M,
\end{split}
\end{align}
with a positive function $\phi\in C_{loc}^{4+\alpha,2+\frac{\alpha}{2}}(\partial M\times[0,+\infty))$ such that $\phi(q,0)=1$, $\phi_t(q,0)=\phi_{tt}(q,0)=0$ for $q\in \partial M$, $\phi_t\geq0$ on $\partial M\times[0,+\infty)$ and $\liminf_{t\to\infty}\inf_{\partial M}\phi=\infty$. Assume $\phi$ satisfies $(\ref{ineqn_bdvcond1})$. Then $u\to u_{\infty}$ in $C_{loc}^{4}(M^{\circ})$ as $t\to +\infty$, where $u_{\infty}$ is the solution to the Loewner--Nirenberg problem $(\ref{equn_scalarcurvatureequn})-(\ref{equn_blowingupboundarydata})$.
\end{lem}
 \begin{proof}
 By Lemma \ref{lem_longtimeexistence}, we have that there exists a positive solution $u\in C_{loc}^{4+\alpha,2+\frac{\alpha}{2}}(M\times[0,\infty))$ to the Cauchy--Dirichlet problem $(\ref{equn_scalar20})$. By maximum principle, we have that $u\geq 1$ in $M\times[0,+\infty)$ and $u>1$ for $t$ large. As in Lemma \ref{lem_monotonocityscalar1},  the solution $u$ to $(\ref{equn_scalar20})$ is increasing. Indeed, let $v=u_t$, and then $v$ satisfies
\begin{align*}
&v_t=\frac{4(n-1)}{n-2}\Delta_{g} v+n(n-1)\big(1-\frac{n+2}{n-2}u^{\frac{4}{n-2}}\big)v,\,\,\,\text{in}\,\,M\times [0,+\infty),\\
&v=0,\,\,p\in M,\\
&v(q,t)=\phi_t(q,t)\geq0,\,\,q\in\,\partial M.
\end{align*}
By maximum principle, $v\geq0$ for $t\geq0$ and $v>0$ for $t$ large. Therefore, $u$ is increasing in $t$. By the uniform upper bound estimates on any compact sub-domain of $M^{\circ}$, we have that $u$ converges pointwisely to $u_{\infty}>0$ locally in $M^{\circ}$. Using the Harnack inequality on $v$ we have that $u$ converges to $u_{\infty}$ locally uniformly in $M^{\circ}$ as $t\to +\infty$. By the standard parabolic estimates, the convergence is in $C_{loc}^{4}(M^{\circ})$. By Lemma \ref{lem_intupperboundestimates} and Lemma \ref{lem_asymlowbdnearbd1}, there exists $C>0$ such that
\begin{align*}
\frac{1}{C}x^{-\frac{n-2}{2}}\leq u_{\infty}\leq C x^{-\frac{n-2}{2}},
\end{align*}
in a neighborhood of $\partial M$, where $x$ is the distance function to $\partial M$.
\end{proof}

\begin{lem}\label{lem_genc}
Assume $(M,g)$ is a compact Riemannian manifold with boundary of class $C^{4,\alpha}$ and $R_g=-n(n-1)$. For any $u_0\in C^{2,\alpha}(M)$ such that $u_0>1$, there exists a positive function $\phi\in C_{loc}^{2+\alpha,1+\frac{\alpha}{2}}(M,[0,+\infty))$ satisfying the compatibility condition $(\ref{equn_compatibleequn1-2})$ with $u_0$, and satisfying $(\ref{ineqn_bdvcond1})$ as $t\to\infty$ such that there exists a unique positive solution $u\in C_{loc}^{2+\alpha,1+\frac{\alpha}{2}}(M\times[0,+\infty))$ to $(\ref{equn_scalar2-1})$, and $u$ converges to $u_{\infty}$ in $C_{loc}^2(M^{\circ})$ as $t\to+\infty$, where $u_{\infty}$ is the solution to the Loewner--Nirenberg problem $(\ref{equn_scalarcurvatureequn})-(\ref{equn_blowingupboundarydata})$.
\end{lem}
\begin{proof}
We first construct the boundary data $\phi\in C_{loc}^{2+\alpha,1+\frac{\alpha}{2}}(\partial M\times [0,\infty))$. Pick up a positive number $t_1>0$. By Theorem \ref{thm_bddataextension1}, we let $\phi$ satisfy the compatibility condition $(\ref{equn_compatibleequn1-2})$ at $t=0$, $\phi_t\geq 0$ for $t\geq t_1$, $\phi\to \infty$ as $t\to \infty$, and $\phi$ satisfy $(\ref{ineqn_bdvcond1})$ as $t\to\infty$. Since $u_0$ is not necessarily a subsolution of $(\ref{equn_scalarcurvatureequn})$, $\phi_t$ is not necessarily positive at $t=0$. We choose $\phi$ so that $\phi\geq \frac{1}{2} (1+\inf_{M}u_0)$ for $t\in[0,t_1]$. By Lemma \ref{lem_longtimeexistence}, we have that there exists a positive solution $u\in C_{loc}^{2+\alpha,1+\frac{\alpha}{2}}(M\times[0,\infty))$  to the Cauchy--Dirichlet problem $(\ref{equn_scalar20})$. By maximum principle, we have that $u\geq1$ on $M\times[0,\infty)$.

Now we pick up a function $\tilde{\phi}\in C_{loc}^{4+\alpha,2+\frac{\alpha}{2}}(\partial M\times[0,\infty))$ such that $\tilde{\phi}(q,0)=1$, $\tilde{\phi}_t(q,0)=\tilde{\phi}_{tt}(q,0)=0$ on $\partial M$, moreover, $\tilde{\phi}_t\geq0$ and $\tilde{\phi}\leq\phi$ on $\partial M\times[0,\infty)$, and $\tilde{\phi}$ satisfies $(\ref{ineqn_bdvcond1})$ as $t\to\infty$ (for this we can just let $\tilde{\phi}=\phi$ for $t$ large).  By Lemma \ref{lem_generalconformalclassconvergence1}, for the initial data $\tilde{u}_0=1$ and the boundary data $\tilde{\phi}$, there exists a unique positive solution $\tilde{u}\in C_{loc}^{4+\alpha,2+\frac{\alpha}{2}}(M\times[0,\infty))$ to the Cauchy--Dirichlet problem $(\ref{equn_scalar20})$, and $\tilde{u}$ converges to $u_{\infty}$ in $C_{loc}^{4}(M^{\circ})$ as $t\to\infty$, where $u_{\infty}$ is the solution of the Loewner--Nirenberg problem $(\ref{equn_scalarcurvatureequn})-(\ref{equn_blowingupboundarydata})$. By maximum principle, we have that $\tilde{u}\geq 1$ in $M\times[0,+\infty)$.

Now let $\xi=u-\tilde{u}$. Then $\xi$ satisfies
\begin{align*}
&\xi_t=\frac{4(n-1)}{n-2}\Delta_{g} \xi+n(n-1)\zeta\xi,\,\,\,\text{in}\,\,M\times [0,+\infty),\\
&\xi(p,0)=u_0(p)-1>0,\,\,p\in M,\\
&\xi(q,t)=\phi(q,t)-\tilde{\phi}(q,t)\geq0,\,\,q\in\,\partial M,
\end{align*}
where $\zeta=\frac{\big((\tilde{u}-\tilde{u}^{\frac{n+2}{n-2}})-(u-u^{\frac{n+2}{n-2}})\big)}{\tilde{u}-u}$ when $\tilde{u}\neq u$, and $\zeta=1-\frac{n+2}{n-2}u^{\frac{4}{n-2}}$ otherwise, and hence $\zeta<0$. By maximum principle, $ u\geq \tilde{u}$ in $M\times[0,+\infty)$.

The upper bound estimates is the same as the case $R_g\geq 0$. Notice that $u_{\infty}>1$ by maximum principle of the equation $(\ref{equn_normalizedYamabeequation1})$. Let $\eta=u-u_{\infty}$. Then $\eta(q,t)\to -\infty$ as $q\to \partial M$ for any $t>0$. Notice that $\eta$ satisfies the equation
\begin{align*}
\eta_t=\frac{4(n-1)}{n-2}\Delta_{g} \eta+n(n-1)\beta\eta,\,\,\,\text{in}\,\,M\times [0,+\infty),
\end{align*}
where $\beta=\frac{(u-u^{\frac{n+2}{n-2}})-(u_{\infty}-u_{\infty}^{\frac{n+2}{n-2}})}{u-u_{\infty}}$ when $u\neq u_{\infty}$ and $\beta=1-\frac{n+2}{n-2}u^{\frac{4}{n-2}}$ otherwise, and hence, $\beta<0$. Therefore, by the same argument as the case $R_g\geq0$, we have that $\displaystyle\limsup_{t\to+\infty}\sup_{q\in M^{\circ}}\eta(q,t)\leq0$. Therefore, $u(\cdot,t)\to u_{\infty}(\cdot)$ locally uniformly on $M^{\circ}$ as $t\to\infty$. By standard interior estimates for the parabolic equation, we have that $\tilde{u}\to u_{\infty}$ in $C_{loc}^{2}(M^{\circ})$ as $t\to +\infty$.
\end{proof}

\begin{proof}[Proof of Theorem \ref{thm_generaldirectflowdiscussion1}]
By the argument in Appendix \ref{Appendix1}, there exists a positive solution $\bar{u}\in C^{2,\alpha}(M)$ to $(\ref{equn_scalarcurvatureequn})$ such that $v_0<\bar{u}<u_0$ on $M$. Let $\bar{g}=\bar{u}^{\frac{4}{n-2}}g$ and hence $R_{\bar{g}}=-n(n-1)$. Let $\eta=\frac{u_0}{\bar{u}}$. Then by Lemma \ref{lem_genc}, we have that there exists a positive function $\phi\in C_{loc}^{2,\alpha}(\partial M\times [0,+\infty))$ such that $\phi$ satisfies the compatibility condition and growth condition of the flow
\begin{align*}
&u_t=\frac{4(n-1)}{n-2}\Delta_{\bar{g}} u+n(n-1)\big(u-u^{\frac{n+2}{n-2}}\big),\,\,\,\text{in}\,\,M\times [0,+\infty),\\
&u(p,0)=\eta(p),\,\,p\in M,\\
&u(q,t)=\phi(q,t),\,\,q\in\,\partial M,
\end{align*}
and there exists a unique positive solution $u$, which converges to the Loewner--Nirenberg solution $u_{\infty}$ in the conformal class in $C_{loc}^2(M^{\circ})$. That is to say, there exists a direct flow starting from $g_0=u_0^{\frac{4}{n-2}}g$ and converging to the Loewner--Nirenberg metric.  Notice that if the first Dirichlet eigenvalue satisfies $\lambda_1(L_g)\leq0$, by Lemma \ref{lem_Escobar1}, it is equivalent to say that $\displaystyle\inf_{u\in C^1(M),\,u\not\equiv 0\,\text{on}\,\partial M}Q(u)=-\infty$ when $g$ is smooth, where the energy $Q(u)$ is defined in $(\ref{equn_energy123})$.
\end{proof}

\section{Positive scalar curvature metrics on compact manifolds with boundary}\label{section3}

In this section, we present examples for conformal classes on compact manifold with boundary that admit positive scalar curvature metrics. Let $M$ be a compact smooth manifold with boundary $\partial M$. It is well-known that there is not topological obstruction for the existence of positive scalar curvature metrics on $M$. For any given smooth Riemannian metric $g$ on $M$, we can extend $(M,g)$ to a smooth closed manifold $(N,h)$. Then by the well known theorems by Kazdan and Warner, Theorem 3.3 in \cite{KW1} and Theorem 6.2 in \cite{KW2}, we have that there exists a diffeomorphism $F:\,N\to N$ and a smooth positive function $u\in N$, such that the scalar curvature of $u^{\frac{4}{n-2}}F^*h$ is positive on $M$ where $F^*h$ is the pull back metric of $h$ by the map $F$. Notice that $F^*h$ is not necessarily in the conformal class $[g]$ on $M$. Another interesting result obtained by Shi-Wang-Wei (Theorem 1.1 in \cite{ShiWangWei}) recently, answering a question by Gromov, states that, any smooth Riemannian metric $h$ on $\partial M$ can be extended to a smooth Riemannian metric $g$ with positive scalar curvature on $M$.

Now we consider the existence of positive scalar curvature metrics in a conformal class on a compact manifold with boundary.

Let $(M^n,h)$ be a smooth compact Riemannian manifold with boundary $\partial M$ of dimension $n\geq4$. Let $g$ be a complete metric equipped on the interior $M^{\circ}$ such that $x^2g$ extends in $C^{k,\alpha}$ up to the boundary, where $x$ is the distance function to the boundary under the metric $h$, then we call $(M^{\circ}, g)$ conformally compact of class $C^{k,\alpha}$. If moreover, $g$ is Einstein, we call $(M^{\circ}, g)$ conformally compact Einstein  (CCE). In the well known paper \cite{Qing}, Qing shows that for a conformally compact Einstein manifold which has conformal infinity of positive Yamabe constant, there exists a conformal compactification $\bar{g}$ with positive scalar curvature. Definitely, the CCE metric $g$ is the Loewner--Nirenberg metric in the conformal class of $\bar{g}$, and hence the flow here connects these two metrics with the starting metric $\bar{g}$ and end at $g$ as $t\to+\infty$.

To continue, we now introduce a well-known Neumann type boundary value problem on $(M^n,g)$ introduced by Escobar (\cite{Escobar1}, see also \cite{Cherrier}), which is sometimes called the second type Escobar-Yamabe problem:
\begin{align}\label{equn_EY2}
\begin{split}
&-\frac{4(n-1)}{(n-2)}\Delta_g u +R_gu=0,\,\,\text{in}\,\,M,\\
&\frac{4(n-1)}{n-2}\frac{\partial}{\partial n_g}u+2(n-1)H_gu=2(n-1)u^{\frac{n}{n-2}},\,\,\text{on}\,\,\partial M,
\end{split}
\end{align}
where $n_g$ is the unit out-normal vector field on $\partial M$ and $H_g$ is the mean curvature on the boundary. The solution of problem $(\ref{equn_EY2})$ is a critical point of the energy functional:
 \begin{align}\label{equn_energy123}
 Q(u)=\frac{\int_M(|\nabla u|^2+\frac{n-2}{4(n-1)}R_gu^2)dV_g+\frac{n-2}{2}\int_{\partial M}H_g u^2 dS}{\big(\int_{\partial M}|u|^{\frac{2(n-1)}{n-2}}dS\big)^{\frac{n-2}{n-1}}},
 \end{align}
 for any $u\in C^1(M)$, where $H_g$ is the mean curvature of the boundary $\partial M$ on $(M, g)$, $dV_g$ is the volume element of $(M,g)$ and $dS$ is the volume element of $(\partial M, g\big|_{\partial M})$. Notice that the functional does not always have a finite lower bound in a general conformal class, as pointed by Zhiren Jin, see \cite{Escobar3}. For instance, $\displaystyle\inf_{u\in C^1(M),\,u\not\equiv 0\,\text{on}\,\partial M}Q(u)=-\infty$ when $(M,g)$ is obtained by deleting a small geodesic ball on a closed Riemannian manifold with negative scalar curvature. Under the assumption $\displaystyle\inf_{u\in C^1(M),\,u\not\equiv 0\,\text{on}\,\partial M}Q(u)>-\infty$ and the positive mass theorem, the problem has been solved: it was studied in \cite{Escobar3}\cite{Escobar5}\cite{Marques1}\cite{Marques2}\cite{ChenS}\cite{Almaraz} assuming the positive mass theorem or the Weyl tensor $W_g=0$ on $M$, and the remaining cases ($n\geq 6$, $\partial M$ is umbilic, $\displaystyle\inf_{u\in C^1(M),\,u\not\equiv 0\,\text{on}\,\partial M}Q(u)>0$ and the subset on $\partial M$ where the vanishing of the Weyl tensor of $g$ is of certain order is the whole boundary $\partial M$) were recently solved in \cite{Mayer-Ndiaye} without the positive mass theorem.  In particular, for manifolds of dimension $3\leq n\leq 7$ (see \cite{SchoenYau}\cite{SchoenYau1}\cite{SchoenYau2}), or spin manifolds \cite{Witten}, positive mass theorem holds. Recently, Schoen-Yau \cite{SchoenYau0} presented a proof that positive mass theorem holds true in general dimension.  For positive mass theorem and the existence theory of solutions to the Escobar-Yamabe problem, for the regularity of the metric, it is enough that the metric is of class $C^k$ for some sufficiently large $k>0$.

 Recall that by an easy computation, Escobar observed in \cite{Escobar3} the following lemma.
\begin{lem}\label{lem_Escobar1}(Escobar, \cite{Escobar3})
Assume $(M^n,g)$ is a smooth compact Riemannian manifold with boundary. Let $\lambda_1(L_g)$ be the first eigenvalue of the Dirichlet boundary value problem of the conformal Laplacian $L_g=-\frac{4(n-1)}{n-2}\Delta_g+R_g$. Then $\displaystyle\inf_{u\in C^1(M),\,u\not\equiv 0\,\text{on}\,\partial M}Q(u)>-\infty$ if and only if $\lambda_1(L_g)>0$.
\end{lem}
When $R_g\geq 0$, we know that $\lambda_1(L_g)>0$, and hence $\displaystyle\inf_{u\in C^1(M),\,u\not\equiv 0\,\text{on}\,\partial M}Q(u)>-\infty$. We will use a perturbation argument to show that the other direction holds.
\begin{thm}\label{thm_pscalar1}
Let $(M,g)$ be a smooth compact Riemannian manifold with boundary. Then there exists a positive scalar curvature metric in the conformal class of $(M, g)$ if and only if $\displaystyle\inf_{u\in C^1(M),\,u\not\equiv 0\,\text{on}\,\partial M}Q(u)>-\infty$.
\end{thm}
\begin{proof}
Necessity is a trivial consequence of Lemma \ref{lem_Escobar1}. Now we show the other direction. {\bf Claim.} If there is a metric $h\in [g]$ such that $R_h=0$ on $M$, then there exists a metric $h_1\in [g]$ such that $R_{h_1}>0$ on $M$. This is an application of the implicit function theorem as proof of Theorem 6.11 in \cite{Aubin}.  Consider the Dirichlet boundary value problem
\begin{align*}
&-\frac{4(n-1)}{n-2}\Delta_h u=f u^{\frac{n+2}{n-2}},\,\,\,\text{in}\,M,\\
&u=1,\,\,\text{on}\,\partial M,
\end{align*}
for any given function $f\in C^{\alpha}(M)$. We want to obtain a positive solution to the  problem. Let $\mathcal{C}_0=\{u\in C^{2,\alpha}( M):\,\,u\big|_{\partial M}=0\}$. Define the map $F:\,\mathcal{C}_0\times C^{\alpha}(M)\to C^{\alpha}(M)$ such that
\begin{align*}
F(u, f)=-\frac{4(n-1)}{n-2}\Delta_h (1+u)-f (1+u)^{\frac{n+2}{n-2}}.
\end{align*}
This is a $C^1$ map. We take derivative
\begin{align*}
D_uF(v,f)=-\frac{4(n-1)}{n-2}\Delta_h v- \frac{n+2}{n-2}f (1+u)^{\frac{4}{n-2}}v.
\end{align*}
Recall that $F(0,0)=0$ since $R_h=0$. At $(0,0)$,
\begin{align*}
D_uF(v,0)=-\frac{4(n-1)}{n-2}\Delta_h v
\end{align*}
is invertible. By implicit function theorem, there exists a constant $\epsilon >0$ small such that for any $f\in C^{\alpha}(M)$ with $\|f\|_{C^{\alpha}(M)}\leq\epsilon$, we have that the Dirichlet problem has a unique solution $\tilde{u}=1+u$ in the neighborhood of $1$ in $C^{2,\alpha}(M)$. Take $f$ to be a small positive constant. We have that $\tilde{u}>0$ on $M$. This completes the proof of {\bf Claim}. The solution of Escobar-Yamabe problem $(\ref{equn_EY2})$, where positive mass theorem is used for certain cases, tells that when $\displaystyle\inf_{u\in C^1(M),\,u\not\equiv 0\,\text{on}\,\partial M}Q(u)>-\infty$, there exists a metric in the conformal class with zero scalar curvature, and hence by {\bf Claim}, we have that there exists a positive scalar curvature metric in the conformal class. This completes the proof of the Theorem.

 Now we give a completely different proof of the Theorem, which does not use the solution of the Escobar-Yamabe problem $(\ref{equn_EY2})$, and hence avoids using the positive mass theorem. By the assumption $\displaystyle\inf_{u\in C^1(M),\,u\neq 0\,\text{on}\,\partial M}Q(u)>-\infty$ and Lemma \ref{lem_Escobar1}, we have $\lambda_1(L_g)>0$, and hence $L_g$ is invertible. Let $\phi_1$ be the first eigenfunction with respect to $\lambda_1(L_g)$. Then we can take $\phi_1>0$ in $M^{\circ}$ and $\phi=0$ on $\partial M$. By standard elliptic theory (see Theorem 8.6 in \cite{GT}), there exists a unique solution $\xi\in C^{\infty}(M)$ to the boundary value problem
 \begin{align*}
 &-\frac{2(n-1)}{n-2}\Delta_g\xi+R_g\xi=0,\,\text{in}\,M,\\
 &\xi=1,\,\,\,\,\,\,\,\text{on}\,\partial M.
 \end{align*}
 By continuity, $\xi>\frac{1}{2}$ in a neighborhood of $\partial M$. Therefore, there exists a constant $C_1>0$ such that $u\triangleq 2(\xi + C_1 \phi_1)>1$ on $M$. Let $g_1=u^{\frac{4}{n-2}}g$. Then we have $R_{g_1}\geq0$. By similar perturbation argument as above, we have that there exists a metric $g_2$ in the conformal class such that $R_{g_2}>0$ on $M$.
\end{proof}

Remark on the proof of Lemma \ref{lem_Escobar1}: In \cite{Escobar3}, Escobar showed that when $\lambda_1(L_g)<0$, then $\displaystyle\inf_{u\in C^1(M),\,u\neq 0\,\text{on}\,\partial M}Q(u)=-\infty$. When $\lambda_1(L_g)=0$, similar argument derives that $\displaystyle\inf_{u\in C^1(M),\,u\neq 0\,\text{on}\,\partial M}Q(u)=-\infty$. Indeed, since the boundedness of $\displaystyle\inf_{u\in C^1(M),\,u\neq 0\,\text{on}\,\partial M}Q(u)$ is conformally invariant, we pick up $g$ in the conformal class so that $R_g=-n(n-1)$ (for existence see Lemma \ref{lem_Dirichletbvpscequn1}). Let $\phi_1$ be the first eigenfunction with respect to $\lambda_1(L_g)=0$. Then $\phi_1>0$ in $M^{\circ}$ and $\phi_1=0$ on $\partial M$. For any constant $\epsilon>0$, let $u_{\epsilon}=\phi_1+\epsilon$. It is easy to check that $Q(u_{\epsilon})\to -\infty$ as $\epsilon\searrow0$. When $\lambda_1(L_g)>0$, as the argument in the last paragraph in the proof of Theorem \ref{thm_pscalar1}, there exists $g_1\in[g]$ such that $R_{g_1}\geq0$. Therefore,
for any $u\in C^1(M)$ such that $u\not\equiv 0$ on $\partial M$, we have that
\begin{align*}
Q(u)&=\frac{\int_M(|\nabla u|_{g_1}^2+\frac{n-2}{4(n-1)}R_{g_1}u^2)dV_{g_1}+\frac{n-2}{2}\int_{\partial M}H_{g_1} u^2 dS}{\big(\int_{\partial M}|u|^{\frac{2(n-1)}{n-2}}dS\big)^{\frac{n-2}{n-1}}}\\
&\geq \frac{\frac{n-2}{2}\int_{\partial M}H_{g_1} u^2 dS}{\big(\int_{\partial M}|u|^{\frac{2(n-1)}{n-2}}dS\big)^{\frac{n-2}{n-1}}}\\
&\geq -C,
\end{align*}
where the last inequality is by H$\ddot{\text{o}}$lder's inequality, and $C>0$ depends only on $\sup_{\partial M}|H_{g_1}|$ and volume of $\partial M$ under the restriction of $g_1$. Therefore, $\displaystyle\inf_{u\in C^1(M),\,u\neq 0\,\text{on}\,\partial M}Q(u)>-\infty$. Combining these three cases we have that $\lambda_1(L_g)>0$ iff $\displaystyle\inf_{u\in C^1(M),\,u\neq 0\,\text{on}\,\partial M}Q(u)>-\infty$.

Remark. The perturbation argument in Theorem \ref{thm_pscalar1} and Lemma \ref{lem_Escobar1} show that for a smooth compact Riemannian manifold $(M,g)$, if $\displaystyle\inf_{u\in C^1(M),\,u\not\equiv 0\,\text{on}\,\partial M}Q(u)=-\infty$, then there exists no conformal metric $g_1=u^{\frac{4}{n-2}}g$ with $u\in C^{2,\alpha}(M)$ such that $R_{g_1}=0$. This gives a way to see the known result that  $\displaystyle\inf_{u\in C^1(M),\,u\not\equiv 0\,\text{on}\,\partial M} Q(u) >-\infty$ is a necessary condition for the existence of a solution to the Yamabe problem studied by Escobar in \cite{Escobar2}.

Remark. Let $\lambda_1$ be the first eigenvalue of the problem studied by Escobar in \cite{Escobar2}
  \begin{align*}
\begin{split}
&-\frac{4(n-1)}{(n-2)}\Delta_g u +R_gu=0,\,\,\text{in}\,\,M,\\
&\frac{\partial}{\partial n_g}u+\frac{n-2}{2}H_gu+\lambda_1(B) u=0,\,\,\text{on}\,\,\partial M.
\end{split}
\end{align*}
 It is known that $\lambda_1(B)$ could be $-\infty$, and $\lambda_1(B)>-\infty$ iff $\displaystyle\inf_{u\in C^1(M),\,u\neq 0\,\text{on}\,\partial M}Q(u)>-\infty$, see \cite{Escobar3}. But for the proof of this fact, both the argument using H$\ddot{\text{o}}$lder's inequality directly in \cite{Escobar3} and the argument at the beginning of Proposition 2.1 in \cite{Escobar2} (using Proposition 1.1 and Proposition 1.4) seem to prove the same direction: If $\lambda_1(B)>-\infty$, then $\displaystyle\inf_{u\in C^1(M),\,u\neq 0\,\text{on}\,\partial M}Q(u)>-\infty$. For the other direction, one uses the argument in the above remark on the proof of Lemma \ref{lem_Escobar1} to obtain that there exists $g_1\in[g]$ such that $R_{g_1}\geq0$. Then
 \begin{align*}
 \lambda_1(B)&=\inf_{u\in C^1(M),\,u\neq 0\,\text{on}\,\partial M}\frac{\int_M(|\nabla u|_{g_1}^2+\frac{n-2}{4(n-1)}R_{g_1}u^2)dV_{g_1}+\frac{n-2}{2}\int_{\partial M}H_{g_1} u^2 dS}{\int_{\partial M}u^2dS}\\
 &\geq \inf_{u\in C^1(M),\,u\neq 0\,\text{on}\,\partial M}\frac{\frac{n-2}{2}\int_{\partial M}H_{g_1} u^2 dS}{\int_{\partial M}u^2dS}\\
 &\geq -\frac{(n-2)}{2}\sup_{\partial M}|H_{g_1}|.
 \end{align*}
It is clear that by the remark on the proof of Lemma \ref{lem_Escobar1} and the second approach in the proof of Theorem \ref{thm_pscalar1}, $(M^n,g)$ is only required to be of class $C^{2,\alpha}$ in Lemma \ref{lem_Escobar1} and Theorem \ref{thm_pscalar1}.

\section{The Yamabe flow-- A conformally invariant flow}

We consider the (renormalized) Yamabe flow
\begin{align}\label{equn_Yamabeflow1}
u_t=(n-1)u^{-\frac{4}{n-2}}(\Delta_gu-\frac{n-2}{4(n-1)}(R_gu+n(n-1)u^{\frac{n+2}{n-2}})),
\end{align}
which can also be written as
\begin{align}
(u^{\frac{n+2}{n-2}})_t=\frac{(n-1)(n+2)}{n-2}(\Delta_gu-\frac{n-2}{4(n-1)}(R_gu+n(n-1)u^{\frac{n+2}{n-2}})).
\end{align}
The flow is conformally covariant in the sense that, under the conformal change $g=\varphi^{\frac{4}{n-2}}h$, the above equation $(\ref{equn_Yamabeflow1})$ becomes
\begin{align*}
v_t=(n-1)v^{-\frac{4}{n-2}}(\Delta_hv-\frac{n-2}{4(n-1)}(R_hv+n(n-1)v^{\frac{n+2}{n-2}})),
\end{align*}
where $v=u\varphi$. So without loss of generality, let $g$ be the metric on $M$ in the conformal class such that $R_g=-n(n-1)$. We consider the Cauchy--Dirichlet problem $(\ref{equn_scalarflow2})$ where $u_0\in C^{2k+2,\alpha}(M)$ and $\phi(q,t)\in C_{loc}^{2k+2+\alpha,k+1+\frac{\alpha}{2}}(\partial M\times[0,+\infty))$ with $k=0,\,1$ satisfying the $C^{2+\alpha,1+\frac{\alpha}{2}}$ compatibility condition
 \begin{align}\label{equn_compatibleequn1-2-22}
u_0(p)=\phi(p,0),\,\,\phi(p,0)^{\frac{4}{n-2}}\phi_t(p,0)=(n-1)[\Delta_g u_0(p)+\frac{n(n-2)}{4}(u_0(p)-u_0(p)^{\frac{n+2}{n-2}})],
\end{align}
for $p\in\,\partial M$, and in addition the $C^{4+\alpha,2+\frac{\alpha}{2}}$ compatibility condition
\begin{align}\label{equn_compatibleequn4-2-22}
\phi(p,0)^{\frac{4}{n-2}}\phi_{tt}(p,0)+\frac{4}{n-2}\phi(p,0)^{\frac{6-n}{n-2}}(\phi_t(p,0))^2=(n-1)L(\mu),
\end{align}
where the linear operator $L$ is defined as
\begin{align*}
L(\mu)=\Delta_g\mu +\frac{n(n-2)}{4}(1-\frac{n+2}{n-2}u_0(p)^{\frac{4}{n-2}})\mu,
\end{align*}
and the function $\mu$ is defined as
\begin{align}\label{equn_CVterms}
\mu(p)=(n-1)u_0(p)^{-\frac{4}{n-2}}[\Delta_g u_0(p)+\frac{n(n-2)}{4}(u_0(p)-u_0(p)^{\frac{n+2}{n-2}})].
\end{align}

\begin{lem}\label{lem_flow2existinglongtime}
Assume $(M^n,g)$ is a compact Riemannian manifold with boundary of class $C^{4,\alpha}$ such that $R_g=-n(n-1)$. Let $u_0\in C^{2k+2,\alpha}(M)$ be a positive function, and $\phi\in C_{loc}^{2k+2+\alpha,k+1+\frac{\alpha}{2}}(M\times[0,\infty))$ be a positive function satisfying the compatibility condition $(\ref{equn_compatibleequn1-2-22})$ for $k=0$ and $(\ref{equn_compatibleequn4-2-22})$ in addition for $k=1$. Then there exists a unique positive solution $u\in C_{loc}^{2k+2+\alpha,k+1+\frac{\alpha}{2}}(M\times[0,\infty))$ to the Cauchy--Dirichlet problem $(\ref{equn_scalarflow2})$. Moreover, if $u_0\geq 1$ and $\phi\geq 1$, then $u\geq 1$.
\end{lem}
\begin{proof}
Since $u_0\in C^{2k+2,\alpha}(M)$ is positive, the equation $(\ref{equn_Yamabeflow1})$ is uniform parabolic, and by the compatibility condition on $u_0$ and $\phi$ and the classical theory for semilinear parabolic equations (see Theorem 8.2 in \cite{Lieberman}), there exists $T>0$ such that a positive solution $u$ on $M\times[0,T)$ such that $u\in C^{2k+2+\alpha,k+1+\frac{\alpha}{2}}(M\times [0,T_1])$ for any $T_1<T$ and $k=0,\,1$ respectively. Now for any $0<T_1<T$, by maximum principle,
\begin{align*}
u\geq \min\{1,\,\displaystyle\inf_{M}u_0,\,\inf_{M\times[0,T_1])}\phi\}
\end{align*}
on $M\times [0,T_1]$. In fact, if there exists $(q,t)\in M^{\circ}\times(0,T_1]$, such that $0<u(q,t)=\inf_{M\times[0,T_1]}u<1$, then we have $(u^{\frac{n+2}{n-2}})_t(q,t)\leq0,\,\,\Delta_gu(q,t)\geq 0, (u-u^{\frac{n+2}{n-2}})(q,t)>0$, contradicting with the equation. Also, by similar argument, we have that
\begin{align*}
u\leq\max\{1,\,\sup_{M}u_0,\,\sup_{M\times[0,T_1]}\phi\},
\end{align*}
on $M\times [0,T_1]$. Therefore, by the standard a prior $C^{2k+2+\alpha, 1+k+\frac{\alpha}{2}}$ estimates of parabolic equations, the solution $u>0$ can be extended on $M\times[0,+\infty)$ with $u\in C^{2k+2+\alpha,k+1+\frac{\alpha}{2}}(M\times[0,T])$ for any $T>0$.

\end{proof}

\begin{lem}\label{lem_flow2monotonicity}
Let $(M,g)$, $u_0\in C^{4,\alpha}(M)$ and $\phi\in C_{loc}^{4+\alpha,2+\frac{\alpha}{2}}(M\times[0,\infty))$ be as in Lemma \ref{lem_flow2existinglongtime} with the compatibility condition $(\ref{equn_compatibleequn1-2-22})-(\ref{equn_compatibleequn4-2-22})$. In particular, $R_g=-n(n-1)$. Moreover, let $u_0\geq1$ be a subsolution to the equation $(\ref{equn_scalarcurvatureequn})$. And moreover, we assume $u_0\in C^{4,\alpha}(M)$ and
 \begin{align*}
 L(\mu)\geq0
 \end{align*}
 at the points $q\in \partial M$ such that $\mu(q)=0$ where $L$ and $\mu$ are as in $(\ref{equn_compatibleequn4-2-22})$. Let a positive function $\phi\in C_{loc}^{4+\alpha,2+\frac{\alpha}{2}}(M\times[0,\infty))$ satisfy the compatibility condition $(\ref{equn_compatibleequn1-2-22})$ and $(\ref{equn_compatibleequn4-2-22})$. Also we assume $\phi_t\geq 0$ on $\partial M\times(0,+\infty)$. Then the solution $u$ to $(\ref{equn_scalarflow2})$ satisfies $u_t\geq 0$ on $M\times [0,+\infty)$. In particular, $u(\cdot,t)$ is a sub-solution to $(\ref{equn_scalarcurvatureequn})$ for each $t\geq0$.
\end{lem}
\begin{proof}
Let $v=u_t$, condition on $u_0$ implies $v(q,0)\geq0$. By assumption, we have $\phi\geq 1$ on $\partial M\times[0,+\infty)$ and hence, by Lemma \ref{lem_flow2existinglongtime}, $u\geq 1$ in $M\times[0,\infty)$. Take derivative of $t$ on both sides of the equation in $(\ref{equn_scalarflow2})$, and we have
\begin{align}\label{equn_flowspeedequationincreasing2-2}
\frac{n+2}{n-2}u^{\frac{4}{n-2}}v_t+\frac{4(n+2)}{(n-2)^2}u^{\frac{6-n}{n-2}}v^2=\frac{(n-1)(n+2)}{n-2}(\Delta_gv+\frac{n(n-2)}{4}(1-\frac{n+2}{n-2}u^{\frac{4}{n-2}})v).
\end{align}
By maximum principle, $v$ can not obtain a negative minimum on $M\times[0,T]$ at a point $(q,t)\in M^{\circ}\times(0,T]$, for any $T>0$. Indeed, if otherwise, since $v\geq0$ on $M\times\{0\}\,\bigcup\,\partial M\times[0,\infty)$ and recall that $u\geq1$, by continuity of $v$, there exists $t_1>0$ and $p\in M^{\circ}$ such that
\begin{align*}
v(p,t_1)=\inf_{M\times[0,t_1]}v<0,
\end{align*}
and $|v(p,t_1)|$ is so small that at the point $(p,t_1)$,
\begin{align*}
-\frac{4(n+2)}{(n-2)^2}u^{\frac{6-n}{n-2}}v^2+\frac{n(n-1)(n+2)}{4}(1-\frac{n+2}{n-2}u^{\frac{4}{n-2}})v>0,
\end{align*}
contradicting with the equation $(\ref{equn_flowspeedequationincreasing2-2})$ and the fact $v_t(p,t_1)\leq0$ and $\Delta_gv(p,t_1)\geq0$. Therefore, $v\geq0$ in $M\times[0,+\infty)$.
\end{proof}

\begin{lem}\label{lem_flow2boundarylowerbound}
Let $(M,g)$, $u_0$ and $\phi$ be as in Lemma \ref{lem_flow2monotonicity}. Moreover, assume $\phi$ satisfies that there exists a constant $\beta>0$ such that
\begin{align}\label{inequn_boundarylowerboundcond2-1}
\phi^{-1}\phi_t\leq \beta
\end{align}
on $\partial M\times[0,\infty)$, and
 \begin{align}\label{inequn_boundarylowerboundcond2-2}
 \begin{split}
 &\phi^{\frac{n-1}{2-n}}|\nabla_{g}\phi|\to 0,\\
&\phi^{\frac{n}{2-n}}|\nabla_{g}^2\phi|\to0,\,\,\text{uniformly on}\,\,\,\partial M, \,\,\text{as}\,\,t\to+\infty.
\end{split}
 \end{align}
Let $x$ be the distance function to the boundary on $(M,g)$. Let $U=\{0\leq x \leq x_1\}$.  Then there exist constants $C>0$, $x_1>0$ small and $t_1>0$ large such that the solution $u>0$ to $(\ref{equn_scalarflow2})$ satisfies
\begin{align*}
u(q,t)\geq C(x+\phi^{\frac{2}{2-n}})^{\frac{2-n}{2}}-C
\end{align*}
on $U\times [t_1,+\infty)$.
\end{lem}
\noindent For instance, take $\phi(t)=e^{t},\,t^2e^{t},\,t,$ or any other monotone function of polynomial growth for $t$ large.
\begin{proof}
The proof is similar to Lemma \ref{lem_asymlowbdnearbd1}. Let $f\equiv1$ and $x_1>0$ small be as in \ref{lem_asymlowbdnearbd1} such that the exponential map $F=\text{Exp}:\,\partial M\times[0,x_1]\to\,U$ is a diffeomorphism, and let the barrier function $\varphi\in C_{loc}^{2,\alpha}(U\times[0,\infty))$ be defined in $(\ref{equn_barrierfunctionlowerbound1})$. For $\varphi$ we have the estimates $(\ref{equn_lowerboundnearboundary})$. In particular,
 \begin{align*}
 \varphi^{\frac{4}{n-2}}\varphi_t&=\varphi^{\frac{4}{n-2}}cf^{\frac{2}{n-2}}[(x+(f(q)\phi(q,t))^{\frac{2}{2-n}})^{\frac{-n}{2}}-(x_1+(f(q)\phi(q,t))^{\frac{2}{2-n}})^{\frac{-n}{2}}]\phi^{\frac{n}{2-n}}\,\frac{\partial\phi}{\partial t}\\
 &\leq c^{\frac{n+2}{n-2}}(x+\phi(q,t)^{\frac{2}{2-n}})^{-\frac{n+4}{2}}\phi^{\frac{n}{2-n}}\frac{\partial\phi}{\partial t}\\
 &\leq c^{\frac{n+2}{n-2}}(x+\phi(q,t)^{\frac{2}{2-n}})^{-\frac{n+2}{2}}\phi^{-1}\frac{\partial\phi}{\partial t},
 \end{align*}
 and hence by $(\ref{equn_maintermtestbdr})$ and our condition on $\phi$, for the constants $c>0$ and $x_1$ small enough and $t_1>0$ large enough, we have
 \begin{align*}
 (\varphi^{\frac{n+2}{n-2}})_t\leq\frac{(n-1)(n+2)}{n-2}(\Delta_g\varphi+\frac{n(n-2)}{4}(\varphi-\varphi^{\frac{n+2}{n-2}}))
 \end{align*}
on $U\times[t_1,\infty)$. Now we take $c>0$ small enough so that $u>\varphi$ on $U\times\{t_1\}$. By the definition of $\varphi$, we also know that $u>\varphi$ on $\partial U\times [0,+\infty)$. Let $v=u-\varphi$. Take difference of this inequality and the equation satisfied by $u$ in $(\ref{equn_scalarflow2})$, we have
  \begin{align*}
  u^{\frac{4}{n-2}}v_t+(u^{\frac{4}{n-2}}-\varphi^{\frac{4}{n-2}})\varphi_t\geq \frac{(n-1)(n+2)}{n-2}(\Delta_gv+\frac{n(n-2)}{4}(u-u^{\frac{n+2}{n-2}}-\varphi+\varphi^{\frac{n+2}{n-2}})),
  \end{align*}
on $U\times[t_1,\infty)$. Assume that there exists $(q,t)\in U^{\circ}\times(t_1,+\infty)$ such that $v(q,t)=\displaystyle\inf_{M\times[t_1,t]}v<0$, and then since $u\geq 1$, we have $\varphi(q,t)>1$ and hence at the point $(q,t)$, we have
 \begin{align*}
 &u^{\frac{4}{n-2}}v_t\leq0,\,\,(u^{\frac{4}{n-2}}-\varphi^{\frac{4}{n-2}})\varphi_t\leq 0,\\
 &\Delta_gv\geq0,\,\,\,u-u^{\frac{n+2}{n-2}}-\varphi+\varphi^{\frac{n+2}{n-2}}>0,
 \end{align*}
where $\varphi_t\geq0$ by the assumption $\phi_t\geq0$, contradicting with the above inequality. Therefore, $v\geq 0$ on $U\times[t_1,+\infty)$. This completes the proof of the lemma.

\end{proof}

\begin{proof}[Proof of Theorem \ref{thm_YamabeflowC4settingconvt1}]
By Lemma \ref{lem_flow2existinglongtime} and Lemma \ref{lem_flow2monotonicity}, we have that there exists a positive solution $u$ on $M\times[0,+\infty)$ with $u\in C^{4+\alpha,2+\frac{\alpha}{2}}(M\times[0,T])$ for any $T>0$, and $u_t\geq0$, i.e., $u(\cdot,t)$ is a sub-solution to $(\ref{equn_scalarcurvatureequn})$ for any $t\geq0$ and also, $u\geq1$. By maximum principle, $u(q,t)\leq u_{LN}(q)$ for any $(q,t)\in M^{\circ}\times[0,+\infty)$, where $u_{LN}$ is the solution to the Loewner--Nirenberg problem  $(\ref{equn_scalarcurvatureequn})-(\ref{equn_blowingupboundarydata})$. Alternatively, one can use the local super-solution constructed in Lemma 5.2 in \cite{GSW} to give the local upper bound estimates of $u$ in $M^{\circ}$. Indeed, let $x=x(q)$ be the distance function of $q\in M^{\circ}$ to $\partial M$. There exists $x_1>0$ small such that for $q\in\{0<x(q)\leq x_1\}$, the injectivity radius $i(q)$ at $q$ is larger than $\frac{x(q)}{2}$. We then define a function $\bar{u}$ on $B_{R}(q)$ by
\begin{align*}
\bar{u}(p)=\left(\frac{2R}{R^2-r(p)^2}\right)^{\frac{n-2}{2}}e^{\frac{n-2}{2}(\sqrt{R^2-r(p)^2+\epsilon^2}-\epsilon)}
\end{align*}
for $p\in B_R(q)$ where $R=\frac{x(q)}{2}$, $\epsilon>0$ is some small constant and $r(p)$ is the distance function from $p$ to $q$, and hence $\bar{u}\in C^{2}(B_R(q))$ and $\bar{u}=\infty$ on $\partial B_{R}(q)$. In fact $x_1$ and $\epsilon$ are chosen small enough as in \cite{GSW} so that $\bar{u}$ is a super-solution to $(\ref{equn_scalarcurvatureequn})$ on $B_R(q)$. By maximum principle, $u\leq \bar{u}$ on $B_R(q)$, and hence there exists a uniform constant $C>0$, such that
\begin{align*}
u(q,t)\leq Cx(q)^{\frac{2-n}{2}}
\end{align*}
on $\{0<x(q)\leq x_1\}$ for $t\geq0$.  For any $q\in M-\{0\leq x\leq x_1\}$, taking $R<\min\{\frac{i(q)}{2},x_1\}$ and $\epsilon>0$ small so that $\bar{u}$ is a super-solution  to $(\ref{equn_scalarcurvatureequn})$ on $B_R(q)$, and hence we have that there exists a uniform constant $C>0$ depending on $B_R(q)$ such that $u(p,t)\leq C$ for $p\in B_{\frac{R}{2}}(q)$ and $t\geq0$. Now by standard interior Schauder estimates of parabolic equations, we have that for any compact subset $F\subseteq M^{\circ}$, there exists a uniform constant $C=C(F)>0$ such that
\begin{align*}
\|u\|_{C^{4+\alpha,2+\frac{\alpha}{2}}(F\times[T,T+1])}\leq C
 \end{align*}
 for any $T\geq0$. Since $u$ is locally uniformly bounded from above in $M^{\circ}$ and $u_t\geq0$ on $M\times[0,\infty)$, by Harnack inequality with respect to the equation $(\ref{equn_flowspeedequationincreasing2-2})$ satisfied by $u_t$, we have that $u$ converges locally uniformly in $M^{\circ}$ to a positive function $u_{\infty}$ on $M^{\circ}$ as $t\to+\infty$. By interior Schauder estimates of the uniform parabolic equation in $(\ref{equn_scalarflow2})$, we have that $u\to u_{\infty}$ in $C_{loc}^4(M^{\circ})$, and hence $u_{\infty}$ is a solution to $(\ref{equn_scalarcurvatureequn})$ in $M^{\circ}$.  By the lower bound estimates near the boundary in Lemma \ref{lem_flow2boundarylowerbound} and the above upper bound estimates, there exists a constant $C>0$ such that
 \begin{align*}
 C x^{\frac{2-n}{2}}\geq u_{\infty}\geq \frac{1}{C} x^{\frac{2-n}{2}}
  \end{align*}
  in $\{0<x\leq x_1\}$  for some constant $x_1>0$, where $x$ is the distance function to the boundary on $(M,g)$. Therefore, by uniqueness of the solution to  $(\ref{equn_scalarcurvatureequn})-(\ref{equn_blowingupboundarydata})$, $u_{\infty}=u_{LN}$. This completes the proof of the theorem.

\end{proof}

\begin{appendix}
\section{}\label{Appendix1}

In this section, we show that the homogeneous Dirichlet boundary value problem
\begin{align}
&\label{equn_scalarcurvatureequn2}\frac{4(n-1)}{n-2}\Delta u-R_gu-n(n-1)u^{\frac{n+2}{n-2}}=0,\,\,\,\text{in}\,\,M,\\
&\label{equn_Dirichletbvdata1}u(p)=0,\,\,\,\text{for}\,\,p\in\partial M.
\end{align}
admits a nontrivial solution $u\geq0$ in $M$ when the conformal Laplacian $L=-\frac{4(n-1)}{n-2}\Delta_g +R_g$ has a negative eigenvalue $\lambda_1(L_g)<0$ for the Dirichlet boundary value problem.

Let $(M,g)$ be a compact Riemannian manifold of class $C^{k+2,\alpha}$ with boundary with $k\geq0$. For any given positive function $\varphi_0\in C^{k+2,\alpha}(\partial M)$, by the classical variational method (see Lemma \ref{lem_Dirichletbvpscequn1}), there exists a unique positive solution $u_0\in C^{k+2,\alpha}(M)$ to the Dirichlet boundary value problem
\begin{align}\label{equn_scalar2}
\begin{split}
&\frac{4(n-1)}{n-2}\Delta_g w-R_gw-n(n-1)w^{\frac{n+2}{n-2}}=0,\\
&w\big|_{\partial M}=\varphi_0.
\end{split}
\end{align}
{\bf Claim 1}. If $u_1$ and $u_2$ are the corresponding solutions to $(\ref{equn_scalar2})$ with respect to $\varphi_0=\varphi_1$ and $\varphi_0=\varphi_2$ for two positive functions $\varphi_1\leq\varphi_2$ on $\partial M$, then $u_1\leq u_2$. We now use maximum principle to prove the claim. Let $g_2=u_2^{\frac{4}{n-2}}g$, $\phi=\frac{\varphi_1}{\varphi_2}$ and $v=\frac{u_1}{u_2}$. Then  $v$ satisfies
\begin{align}\label{equn_normalizedYamabeequation1}
&\frac{4(n-1)}{n-2}\Delta_{g_2} v+n(n-1)(v-v^{\frac{n+2}{n-2}})=0,\\
&v\big|_{\partial M}=\phi\leq1. \nonumber
\end{align}
Then by maximum principle, $v$ could not obtain its maximum point with $\displaystyle\sup_{M}v>1$ in $M^{\circ}$, and hence {\bf Claim 1} is proved.

By Lemma \ref{lem_Dirichletbvpscequn1}, we take the metric $g$ such that $R_g=-n(n-1)$ in the conformal class as the background metric. Let $v_j>0$ be the solution to the Dirichlet boundary value problem $(\ref{equn_normalizedYamabeequation1})$ with $g_2=g$ and with the boundary data $\phi=\frac{1}{j}$. Then by Claim 1, $\{v_j\}_{j=1}^{\infty}$ is decreasing as $j$ increases. By standard elliptic estimates we have $\{v_j\}_j$ converges in $C^2(M)$ to a non-negative function $v_0$, i.e., $v_0=\displaystyle\lim_{j\to\infty}v_j$. Then $v_0$ is a non-negative solution to the problem
\begin{align}\label{equn_homogeneousbvp1-1}
\begin{split}
&\frac{4(n-1)}{n-2}\Delta_g v_0+n(n-1)v_0-n(n-1)|v_0|^{\frac{4}{n-2}}v_0=0,\\
&v_0\big|_{\partial M}=0.
\end{split}
\end{align}

We want to show that the limit $v_0$ of $\{v_j\}_j$ is not zero when the first eigenvalue $\lambda_1(L_g)<0$ for the Dirichlet problem of the conformal Laplacian $L_g=-(\frac{4(n-1)}{n-2}\Delta_g +n(n-1))$. For the case $\lambda_1(L_g)<0$, let $\phi_1$ be the first eigenfunction with $1>\phi_1>0$ in $M^{\circ}$. Recall that the minimizer of the energy
\begin{align}\label{equn_energy1-1}
E(u)=\frac{2(n-1)}{n-2}\int_M|\nabla u|_{g}^2+\int_M\frac{n(n-1)}{2}(-u^2+\frac{n-2}{n}|u|^{\frac{2n}{n-2}}),
\end{align}
in the function space
\begin{align}\label{equn_functionspacesetbvp1}
S=\{u\in W^{1,2}(M)\,\big|\,u-\varphi_0\in W_0^{1,2}(M)\}
 \end{align}
 is the unique solution to $(\ref{equn_scalar2})$ when $\varphi_0>0$. Here for the homogeneous Dirichlet problem $(\ref{equn_homogeneousbvp1-1})$, just take $\varphi_0=0$. Then let $\epsilon>0$ be small enough, we have that
\begin{align*}
E(\epsilon \phi_1)&=\frac{2(n-1)}{n-2}\epsilon^2\int_M|\nabla \phi_1|_{g}^2-\frac{n(n-2)}{4}\phi_1^2 dV_{g}+\frac{(n-2)(n-1)}{2}\epsilon^{\frac{2n}{n-2}}\int_M\phi_1^{\frac{2n}{n-2}}dV_{g}\\
&=\epsilon^2[\frac{2(n-1)}{n-2}\int_M|\nabla \phi_1|_{g}^2-\frac{n(n-2)}{4}\phi_1^2 dV_{g}+\frac{(n-2)(n-1)}{2}\epsilon^{\frac{4}{n-2}}\int_M\phi_1^{\frac{2n}{n-2}}dV_{g}].
\end{align*}
Since $\int_M|\nabla \phi_1|_{g}^2-\frac{n(n-2)}{4}\phi_1^2 dV_{g}<0$, there exists $\epsilon>0$, such that $E(\epsilon \phi_1)<0$. Therefore, \begin{align*}
-m\equiv \inf_{v\in W_0^{1,2}(M)}E(v)<0.
\end{align*}
Let $\{u_j\}_{j=0}^{\infty}$ be a minimizing sequence of $E$ on the function space $S$ defined in $(\ref{equn_functionspacesetbvp1})$. Then for $j$ large,
\begin{align*}
\frac{-m}{2}\geq E(u_j)=\frac{2(n-1)}{n-2}\int_M|\nabla u_j|_{g}^2+\int_M\frac{n(n-1)}{2}(-u_j^2+\frac{n-2}{n}|u_j|^{\frac{2n}{n-2}}),
\end{align*}
and hence,
\begin{align*}
\int_M\frac{n(n-1)}{2}u_j^2 \geq \frac{m}{2}+\frac{2(n-1)}{n-2}\int_M|\nabla u_j|_{g}^2+\int_M\frac{(n-2)(n-1)}{2}|u_j|^{\frac{2n}{n-2}}\geq \frac{m}{2}>0.
\end{align*}
Since $u_j\rightharpoonup \bar{v}$ weakly in $W^{1,2}(M)$ sense, by the Sobolev embedding theorem, we have that $u_j\to \bar{v}$ in $L^2(M)$ up to a subsequence. Therefore, $\|\bar{v}\|_{L^2(M)}\geq \sqrt{\frac{m}{2}}>0$, and it is a weak solution to the problem. It is clear that $|\bar{v}|$ is also a minimizer of $E$. By the regularity argument in Appendix B in \cite{Struwe}, $|\bar{v}|\in C^2(M)$. By Harnack inequality, we have that $|\bar{v}|>0$ in $M^{\circ}$ since $\bar{v}$ is not zero, and hence $\bar{v}>0$ in $M^{\circ}$. Therefore, the homogeneous Dirichlet boundary value problem has a non-zero solution $\bar{v}\in C^2(M)\bigcap W_0^{1,2}(M)$ with $\bar{v}>0$ in $M^{\circ}$. By maximum principle, $\bar{v}\leq 1$ in $M$.

On the other hand, by Claim 1, we have $v_0=\lim_{j\to\infty}v_j\geq \bar{v}$ on $M$, and hence $v_0>0$ in $M^{\circ}$. In particular, $v_0$ is the largest solution to $(\ref{equn_homogeneousbvp1-1})$.

In summary, for a general compact Riemannian manifold $(M,g)$ of class $C^{k+2,\alpha}$ with boundary such that $\lambda_1(L_g)<0$, where $\lambda_1(L_g)$ is the first eigenvalue of the conformal Laplacian $L_g$ of the Dirichlet boundary value problem, there exists a largest solution $v_0$ to $(\ref{equn_scalarcurvatureequn2})-(\ref{equn_Dirichletbvdata1})$ such that $v_0>0$ in $M^{\circ}$.

By the convergence of $v_j$ to $v_0$, for any continuous function $u_0>v_0$ on $M$, there exists $j>0$ such that $u_0>v_j$ on $M$.

On a compact manifold $(M,g)$ with boundary of class $C^{2,\alpha}$, if $\lambda_1(L_g)>0$, then by Lemma \ref{lem_Escobar1} and Theorem \ref{thm_pscalar1}, there exists a conformal metric $h\in[g]$ such that $R_h\geq0$. By maximum principle, $v_0=0$ in this case. We do not know if $v_0$ vanishes when $\lambda_1(L_g)=0$.

\section{Boundary data construction}\label{Appendix2}

For any $0< \alpha< 1$ and two integers $n\ge2$ and $k\geq1$, let $(M^n,g)$ be a smooth compact $n$-dimensional Riemannian manifold with boundary $\partial M$ of class $C^{2k,\alpha}$. Here $C^{2k,\alpha}$ is the standard H$\ddot{\text{o}}$lder space. Let $\{f_m\}_{m=0}^{k}$ be a sequence of functions on $M$ such that $f_m\in C^{2(k-m),\alpha}(M)$ for each $0\leq m\leq k$. In this appendix, we will prove that there exists a function $\xi(x,t)\in C^{2k+\alpha,k+\frac{\alpha}{2}}(M\times [0,\epsilon])$ for some $\epsilon>0$ such that $\frac{\partial^m \xi}{\partial t^m}=f_m$ on $\partial M\times\{0\}$ for each $0\leq m\leq k$. This should be a standard result. But since we could not find it in the literature, we give a proof.

\begin{thm}\label{thm_bddataextension1}
Assume $(M^n,g)$ is a compact $n$-dimensional Riemannian manifold with boundary $\partial M$ of class $C^{2k,\alpha}$. Let $\{f_m\}_{m=0}^{k}$ be a sequence of functions on $M$ such that $f_m\in C^{2(k-m),\alpha}(M)$ for each $0\leq m\leq k$. Then there exists a function $\xi(x,t)\in C^{2k+\alpha,k+\frac{\alpha}{2}}(M\times [0,\epsilon])$ for some $\epsilon>0$ such that
\begin{align*}
\frac{\partial^m \xi}{\partial t^m}=f_m
\end{align*}
on $\partial M\times\{0\}$ for each $0\leq m\leq k$. Moreover, we can choose $\xi$ so that at each point $p\in \partial M$ with $f_m(p)=0$ for each $0\leq m\leq k$, we have $\frac{\partial^m}{\partial t^m}\xi(p,t)>0$ for each $0\leq m\leq k$ and $0<t<\epsilon_0$ with some constant $0<\epsilon_0<\epsilon$.
\end{thm}

We start with a lemma.

\begin{lem}\label{lem_bddataextension1}
Assume $(M^n,g)$ is a compact $n$-dimensional Riemannian manifold with boundary $\partial M$ of class $C^{2k,\alpha}$. Let $\{f_m\}_{m=0}^{k}$ be a sequence of functions on $M$ such that $f_m\in C^{2(k-m),\alpha}(M)$ for each $0\leq m\leq k$. Then, there exists a function $u_0\in C^{2k,\alpha}(M)$ such that
\begin{align*}
\Delta_g^mu_0=f_m
\end{align*}
on $\partial M$ for $0\leq m\leq k$.
\end{lem}
\begin{proof}
Let $u_k=f_k$ on $M$. Now for $m=k,k-1,...,1$, we solve the Dirichlet problems inductively
\begin{align*}
&\Delta_g u_{m-1}=u_m,\,\,\text{in}\,\,M,\\
&u_{m-1}=f_{m-1},\,\,\text{on}\,\,\partial M.
\end{align*}
For each problem with index $m$, we obtain a unique solution $u_{m-1}\in C^{2(k+1-m),\alpha}(M)$, for $1\leq m\leq k$, and hence $u_0\in C^{2k,\alpha}(M)$ is just the function we want.
\end{proof}
\begin{proof}[Proof of Theorem \ref{thm_bddataextension1}.] We extend the manifold $(M^n,g)$ to a compact Riemannian manifold $(M_1^n,g_1)$ with boundary $\partial M_1$ of class $C^{2k,\alpha}$ such that $M$ lies in the interior of $M_1$. Hence, $g=g_1\big|_{M}$. Let $u_0\in C^{2k,\alpha}(M)$ be the function constructed in Lemma \ref{lem_bddataextension1}. Now we extend $u_0$ to a function $\tilde{u}_0\in C^{2k,\alpha}(M_1)$ such that $\tilde{u}_0=0$ in a small neighborhood of $\partial M_1$. Hence $u_0=\tilde{u}_0\big|_{M}$.

Now we solve the Cauchy-Dirichlet problem of the heat equation
\begin{align*}
&u_t=\Delta_{g_1}u,\,\,\,\text{in}\,\,\,M_1\times[0,1],\\
&u=0,\,\,\,\,\,\,\text{on}\,\,\,\partial M_1\times[0,1],\\
&u=\tilde{u}_0,\,\,\,\text{on}\,\,\,M_1\times\{0\}.
\end{align*}
It is clear that this linear problem satisfies the compatibility condition since $\tilde{u}_0=0$ in a neighborhood of $\partial M_1$. By standard existence result for Cauchy-Dirichlet problem of the parabolic equation (see \cite{Lieberman}), there exists a unique solution $u\in C^{2k+\alpha,k+\frac{\alpha}{2}}(M_1\times[0,\epsilon])$ for some $\epsilon>0$. Let $\xi_0= u$ on $M\times[0,\epsilon]$. Then $\xi_0 \in C^{2k+\alpha,k+\frac{\alpha}{2}}( M\times [0,\epsilon])$ and there exists a constant $C>0$ and a constant $0<\epsilon_0<\epsilon$ such that the function $\xi\equiv\xi_0+Ct^{k+\frac{\alpha}{2}}$ satisfies that $\frac{\partial^m}{\partial t^m}\xi(p,t)>0$ for $0<t<\epsilon_0$ at the point $p\in \partial M$ such that $f_m(p)=0$ for $0\leq m\leq k$. Hence, $\xi\in C^{2k+\alpha,k+\frac{\alpha}{2}}(M\times [0,\epsilon])$ satisfies that
\begin{align*}
\frac{\partial^m\xi}{\partial t^m}=\Delta_g^m u_0=f_m,
\end{align*}
on $\partial M\times\{0\}$ for each $0\leq m\leq k$. This completes the proof of the theorem.
\end{proof}

\end{appendix}

 \end{document}